\documentclass[aihp]{imsart}

\RequirePackage{amsthm,amsmath,amsfonts,amssymb}
\RequirePackage[numbers,sort&compress]{natbib}
\RequirePackage[colorlinks,citecolor=blue,urlcolor=blue]{hyperref}
\RequirePackage{graphicx}
\usepackage{indentfirst}
\usepackage{caption}
\usepackage{mathrsfs}
\usepackage{xcolor}

\startlocaldefs
\theoremstyle{plain}

\newtheorem{theorem}{Theorem}
\newtheorem{thmaux}{}[section]   
\newtheorem*{theorem 2.12}{Theorem 2.12}
\newtheorem{lemma}[thmaux]{Lemma}
\newtheorem{definition}[thmaux]{Definition}
\newtheorem{example}[thmaux]{Example}
\newtheorem{corollary}[thmaux]{Corollary}
\theoremstyle{remark}
\newtheorem{remark}{Remark}

\endlocaldefs

\begin{document}
\begin{frontmatter}
\title{Most Probable KAM Tori in Stochastic Hamiltonian Systems Driven by Multiplicative Noise}
\runtitle{Most Probable KAM Tori in Stochastic Hamiltonian Systems}

\begin{aug}
	\author[A]{\inits{X.}\fnms{Xinze}~\snm{Zhang}\ead[label=e1]{zhangxz24@mails.jlu.edu.cn}},
	\author[A,B]{\inits{Y.}\fnms{Yong}~\snm{Li}\ead[label=e2]{liyong@jlu.edu.cn}},
	\author[A]{\inits{X.}\fnms{Xue}~\snm{Yang}\ead[label=e3]{xueyang@jlu.edu.cn}}\thanks{corresponding author.}.
	\address[A]{School of Mathematics, Jilin University, ChangChun, People's Republic of China}
	
	\address[B]{Center for Mathematics and Interdisciplinary Sciences,\\    ~~~ Northeast Normal University, ChangChun, People's Republic of China\\
		\printead{e1,e2,e3}}
	
\end{aug}

\begin{abstract}
This paper investigates the effect of state-dependent multiplicative noise on the integrable structure of Hamiltonian systems. Under a local high-dimensional Lamperti-type condition, we derive the Onsager--Machlup functional in the corresponding flattening coordinates. When the diffusion vector fields are Hamiltonian, the resulting action is minimized uniquely by the deterministic Hamiltonian trajectory. Combining this characterization with a finitely differentiable KAM theorem, we prove that, under suitable assumptions, the invariant tori with Diophantine frequencies of an integrable Hamiltonian system persist, in the sense of most probable paths, under a small deterministic perturbation and multiplicative noise. Furthermore, as the intensity of the multiplicative noise tends to zero, we establish a localized large deviation principle that characterizes the asymptotic probability of solution trajectories deviating from these invariant tori.
\end{abstract}

\begin{keyword}[class=MSC]
\kwd[Primary ]{70H08}
\kwd{60H10}
\kwd[; secondary ]{82C35}
\kwd{60F10}
\end{keyword}

\begin{keyword}
\kwd{Stochastic Hamiltonian systems}
\kwd{Invariant tori}
\kwd{Multiplicative noise}
\kwd{Onsager--Machlup functional}
\kwd{Large deviation}
\kwd{KAM theory}
\end{keyword}

\end{frontmatter}

\section*{Statements and Declarations}
\begin{itemize}
	\item Ethical approval\\
	Not applicable.
	
	\item Data Availability Statement.\\
	No datasets were generated or analysed during the current study.
	
	\item The Conflict of Interest Statement. \\
	We have no conflicts of interest to disclose.
	
	\item Funding\\
	The second author (Y. Li) was supported by National Natural Science Foundation of China (Grant No. 12471183, 12531009). The third author (X. Yang) was supported by National Natural Science Foundation of China (Grant No. 12371191).
\end{itemize}

\section{Introduction}

This paper focuses on the persistence of invariant tori in stochastic Hamiltonian systems, particularly examining their stability under random perturbations. We consider the following stochastic Hamiltonian system driven by multiplicative noise in the Stratonovich sense:
\begin{equation}\label{1}
	\mathrm dX_t = J \nabla H(X_t)\, \mathrm dt + \sigma(X_t) \circ \mathrm dW_t,
\end{equation}
where $ t \in [0, T] $, \(X_t \in \mathbb{D} \subset \mathbb{R}^{2n}\), \(\mathbb{D}\) is a nonempty bounded open set, and \(\overline{\mathbb{D}}\) denotes its closure. Here $ H:\overline{\mathbb{D}} \to \mathbb{R} $ is the Hamiltonian, \(J\) is the standard symplectic matrix, \(\nabla H(X)\) denotes the gradient of \(H\), and $ \sigma:\overline{\mathbb{D}} \to \mathbb{R}^{2n \times 2n} $ is the diffusion coefficient.

By integrating the Onsager-Machlup functional, the theory of most probable paths, large deviation principles, and KAM theory, we establish a framework for analyzing the most probable paths and stability of the system under stochastic conditions, thereby revealing the persistence of invariant tori in the sense of most probable paths.

Hamiltonian systems are foundational in classical mechanics, with applications across physics, astronomy, mechanical systems, and more. The core framework, dating back to the early 19th century, was introduced by William Rowan Hamilton. Hamiltonian mechanics describes the state of a system by canonical coordinates \(q\in\mathbb{R}^{n}\) and conjugate momenta \(p\in\mathbb{R}^{n}\). The Hamiltonian function \(H(q,p)\) represents the total energy of the system, including both kinetic and potential energy. The time evolution of the system is governed by Hamilton's equations
\begin{equation}\label{2}
	\mathrm dX_t = J \nabla H(X_t)\, \mathrm dt.
\end{equation}

This formulation not only provides a precise description of system dynamics but also ensures energy conservation and preservation of symplectic geometry. However, when subjected to perturbations, the system’s behavior can become complex. In particular, understanding how to maintain long-term stability under small perturbations presents a critical challenge.

To address these challenges, the KAM theory was a significant breakthrough in the 20th century. Kolmogorov \cite{1} hypothesized that when Hamiltonian systems are subject to small perturbations, some of the invariant tori (regular orbital structures) would persist and avoid chaotic behavior. Arnold \cite{2} and Moser \cite{3} subsequently provided rigorous proofs of this conjecture, formalizing what is now known as the KAM theory. The core result demonstrates that under small perturbations, most of invariant tori continue to exist, preserving the system’s quasi-periodic motions. This result profoundly advanced the study of Hamiltonian systems' stability under perturbations. Subsequent relevant developments are referred to as \cite{4, 5, 7, 8, 9, 10, 13}, and so on.

The original KAM framework was designed primarily for deterministic perturbations, leaving the question of its applicability in stochastic settings unresolved. As a result, validating the extension of the KAM theory under stochastic perturbations and quantifying the stability of invariant tori has emerged as a central problem in stochastic Hamiltonian system research. In recent years, stochastic Hamiltonian systems have attracted substantial attention and have become a rapidly expanding and compelling topic in stochastic dynamics. Wu \cite{15} developed a unified framework for large and moderate deviations in stochastic Hamiltonian systems, yielding quantitative estimates for rare-event probabilities. Li \cite{17} established an averaging principle for completely integrable stochastic Hamiltonian systems, clarifying their long-time behavior under small noise. Davini and Siconolfi \cite{17.5} investigated one-dimensional critical Hamilton-Jacobi equations with stationary ergodic Hamiltonians. Zhang \cite{18} analyzed stochastic Hamiltonian flows and proposed computational methods based on the Bismut formula. For further recent developments, see \cite{18.5, 19, 20, 21, 22} and so on.

The viewpoint adopted in this paper, namely, analyzing the existence and persistence of invariant tori in stochastic Hamiltonian systems in the sense of most probable paths, is novel. This viewpoint was developed for additive-noise Hamiltonian systems in \cite{47} and for a stochastic Schr"odinger equation in \cite{56}. The present paper extends this framework to a class of state-dependent multiplicative diffusions satisfying the local flattening condition (C3). The state-dependent distortion of maximum-norm path tubes is handled by a high-dimensional Lamperti-type local transformation and a direct conditional small-ball analysis. Although several new restrictive assumptions are imposed in our analysis, we believe that this work is of substantial significance.

The main results of this paper are as follows:

\textbf{1) We establish the Onsager-Machlup functional for high-dimensional stochastic ordinary differential equations driven by multiplicative noise in the Stratonovich sense.}

The Onsager-Machlup functional is a fundamental analytical tool for the study of path probabilities and rare events. It characterizes the most probable transition paths among all sufficiently smooth trajectories in a noise-driven system. In a probabilistic framework, it quantifies the relative likelihood of different paths and thus plays a role analogous to that of the action functional in classical mechanics. The Onsager-Machlup functional originated in the work of Onsager \cite{23} and Machlup \cite{24} in 1953, where it was introduced to describe the probability density functional for diffusion processes with linear drift and constant diffusion coefficients. Subsequently, in 1957, Tisza and Manning \cite{25} extended this theory to nonlinear equations, and in the same year, Stratonovich \cite{26} provided a rigorous theoretical framework. Further important developments can be found in \cite{43,44,45,46}, among others.

Most of the existing work on the Onsager-Machlup functional has been carried out in the setting of additive noise. For multiplicative noise, the available results fall roughly into two directions. The first concerns equations in Euclidean spaces. To the best of our knowledge, only one-dimensional results are currently available in this setting, see \cite{57}. The second concerns equations on Riemannian manifolds, see \cite{58,59}. In that framework, the Riemannian metric is prescribed in terms of the diffusion coefficient, and therefore the corresponding results do not cover the Euclidean setting considered here. Inspired by the approach in \cite{57}, we establish the Onsager-Machlup functional for equation \eqref{1} in the high-dimensional case under the restrictive assumption \((C3)\). 

We first state the standing assumptions on \(H\) and \(\sigma\) that will be imposed throughout the paper.
\begin{itemize}
	\item[(C1)] The Hamiltonian function \( H \in C^l(\overline{\mathbb{D}}, \mathbb{R}) \) with $ l > 2(n + 1) $, and its gradient \(\nabla H\) is Lipschitz continuous on \(\overline{\mathbb D}\) with Lipschitz constant \(L\).
	
	\item[(C2)] The diffusion coefficient \( \sigma \in C^2( \overline{\mathbb{D}}, \mathbb{R}^{2n \times 2n}) \) is symmetric. In addition, \( \sigma \) satisfies a uniform ellipticity condition: there exist constants \( 0 < \lambda \leq \Lambda \) such that
	\[
	\lambda |v|^2 \leq v^\top \sigma(x)\sigma(x)^\top v \leq \Lambda |v|^2, \quad \forall v \in \mathbb{R}^{2n}, ~\forall x \in \overline{\mathbb{D}}.
	\]
	Consequently, the pointwise inverse matrix field \(x\mapsto\sigma(x)^{-1}\) is continuous, bounded, and symmetric.
	
	\item[(C3)] There exist open sets \(\widetilde{\mathcal O}\subset\mathbb R^{2n}\) and \(\mathcal O\Subset\mathbb D\), and a \(C^2\)-diffeomorphism \(U:\widetilde{\mathcal O}\to\mathcal O\) onto \(\mathcal O\) such that
	\[
	DU(y)=\sigma(U(y)),\qquad y\in\widetilde{\mathcal O},
	\]
	where \(DU\) denotes the Jacobian matrix of \(U\).
\end{itemize}

The first main result of this paper is stated as follows.
For \(h=(h^1,\ldots,h^{2n})\in C([0,T];\mathbb R^{2n})\), set
\[
 \|h\|_{\infty,*}:=\max_{1\le j\le 2n}\sup_{0\le t\le T}|h^j(t)|.
\]
Fix \(x_0\in\mathcal O\), let \(y_0:=U^{-1}(x_0)\), and let \(\varphi(0)=x_0\) with \(\varphi([0,T])\Subset\mathcal O\). The local maximum tube generated by \(U\) is
\begin{equation}\label{local maximum tube}
 \mathcal K_U(\varphi,\varepsilon)
 :=\left\{x\in C([0,T];\mathcal O):x(0)=x_0,\ 
 \bigl\|U^{-1}\circ x-U^{-1}\circ\varphi\bigr\|_{\infty,*}\le\varepsilon\right\}.
\end{equation}
This is the ordinary maximum tube in the local coordinates generated by \(U\).

\begin{theorem}\label{T1}
	Assume that Conditions \((C1)\)-\((C3)\) hold. Let \(X\) solve \eqref{1} with \(X_0=x_0\in\mathcal O\), and let \(\varphi\) be absolutely continuous with \(\varphi(0)=x_0\), \(\varphi([0,T])\Subset\mathcal O\), and \(U^{-1}\circ\varphi-y_0\in\mathbb H_0\), where \(\mathbb H_0\) is the Cameron--Martin space defined in Section~2. Then
	\begin{equation*}
	 \lim_{\varepsilon\to 0}
	 \frac{\mathbb P\bigl(X\in\mathcal K_U(\varphi,\varepsilon)\bigr)}
	 {\mathbb P\bigl(\|W\|_{\infty,*}\le\varepsilon\bigr)}
	 =\exp\left\{-\frac12\mathcal A(\varphi)\right\},
	\end{equation*}
	where the denominator is independent of \(\varphi\), and
	\begin{equation}\label{OM action}
	 \begin{aligned}
	 \mathcal A(\varphi)
	 =\int_0^T\left|\sigma(\varphi_s)^{-1}
	 \bigl(J\nabla H(\varphi_s)-\dot\varphi_s\bigr)\right|^2\,\mathrm ds
	 -\int_0^T(\operatorname{div}\sigma)(\varphi_s)\cdot
	 \bigl(\sigma(\varphi_s)^{-1}J\nabla H(\varphi_s)\bigr)\,\mathrm ds.
	 \end{aligned}
	\end{equation}
	Here \(\operatorname{div}\sigma=(\operatorname{div}\sigma_1,\ldots,
	\operatorname{div}\sigma_{2n})\), with \(\sigma_\alpha\) denoting the \(\alpha\)-th column of \(\sigma\).
\end{theorem}
\begin{remark}
	In one dimension, (C3) reduces to the ordinary differential equation \(U'(y)=\sigma(U(y))\), which has a local \(C^2\)-diffeomorphic solution whenever \(\sigma\) is sufficiently regular and nonvanishing. In higher dimensions, write \(V_j(x):=\sigma(x)e_j\), \(j=1,\ldots,2n\). Uniform ellipticity makes \((V_1,\ldots,V_{2n})\) a local frame. If these vector fields commute pairwise, \([V_j,V_k]=0\), their local flows commute; the inverse function theorem applied to the joint flow map then gives local coordinates satisfying \(\partial_{y_j}U(y)=V_j(U(y))\). Equivalently, \(DU(y)=\sigma(U(y))\). Thus pairwise commutation is a concrete local sufficient condition for (C3).
\end{remark}

\textbf{2) We prove that the most probable paths of the stochastic Hamiltonian system \eqref{1} coincide exactly with those of the corresponding deterministic Hamiltonian system \eqref{2}.}

The Onsager-Machlup functional quantifies the relative likelihood of continuous trajectories in the path space of a stochastic dynamical system. Its minimizer therefore corresponds to the most probable path \(\hat{\varphi}\) of the stochastic Hamiltonian system \eqref{1}. A standard variational approach is to regard \(OM(\varphi,\dot{\varphi})\) as the Lagrangian \(L\) and apply the Euler-Lagrange equation
\begin{equation*}
	\frac{\partial L}{\partial y}
	-
	\frac{\mathrm{d}}{\mathrm{d}t}
	\left(
	\frac{\partial L}{\partial \dot{y}}
	\right)
	=0.
\end{equation*}
In the present setting, the admissible paths have a fixed initial point and a free terminal point.

In general, however, the resulting variational equation does not yield an explicit most probable path. Consequently, it is difficult to determine whether the most probable dynamics retains the structure of a nearly integrable Hamiltonian system when the multiplicative noise is small. This constitutes a major obstacle to the analysis of stochastic Hamiltonian systems with multiplicative noise. To overcome this difficulty, we impose the following additional assumption:
\begin{itemize}
	\item[(C4)] Each column of the diffusion coefficient \(\sigma\) is a Hamiltonian vector field. More precisely, if we write
	\[
	\sigma(x)=[\sigma_1(x),\dots,\sigma_{2n}(x)],
	\]
	where \(\sigma_\alpha(x)\) denotes the \(\alpha\)-th column vector field, then
	\[
	\sigma_\alpha(x)=J\nabla G_\alpha(x),\quad \alpha=1,\dots,2n.
	\]
	Here \(G_\alpha\in C^2(\mathcal O)\) is a Hamiltonian function.
\end{itemize}

The second main result of this paper is stated as follows.
\begin{theorem}\label{T2}
	Assume that Conditions \((C1)\)-\((C4)\) hold. Fix \(x_0\in\mathcal O\) and let \(\varphi^*\) be the solution of
	\[
	 \dot\varphi_t^*=J\nabla H(\varphi_t^*),\qquad \varphi_0^*=x_0,
	\]
	on \([0,T]\), with \(\varphi^*([0,T])\Subset\mathcal O\). Then \(\varphi^*\) is the unique most probable path with respect to the maximum tubes \(\mathcal K_U(\cdot,\varepsilon)\) among the admissible paths with initial point \(x_0\), in the sense of Definition~\ref{definition MPP}.
\end{theorem}

\textbf{3) We further establish a large deviation principle and identify the corresponding rate function.}

The theory of large deviations concerns the probabilities of rare events in stochastic systems and their exponential rates of decay. For this reason, we consider the small-noise limit and study the following stochastic Hamiltonian system:
\begin{equation}\label{31}
	\mathrm dX_t^\gamma
	=
	J\nabla H(X_t^\gamma)\,\mathrm dt
	+
	\gamma \sigma(X_t^\gamma)\circ \mathrm dW_t,
	\qquad
	X_0^\gamma=x_0 \in \mathbb{D},
	\qquad
	t\in[0,T].
\end{equation}
Here \(W=(W_t)_{0\le t\le T}\) is a \(2n\)-dimensional standard Brownian motion, and \(\gamma>0\) denotes the noise intensity.

Let
\[
 \mathscr C_{x_0}:=\{x\in C([0,T];\mathbb R^{2n}):x(0)=x_0\}
\]
with the uniform topology, and let $ \tau_{\mathbb D}^{\gamma}:=\inf\{t\ge0:X_t^\gamma\notin\mathbb D\} $. For \(\varphi\in\mathscr C_{x_0}\), define
\[
\mathbb I_{\mathbb D}(\varphi)
=
\begin{cases}
	\displaystyle
	\frac12\int_0^{T}
	\left|
	\sigma(\varphi_s)^{-1}\bigl(\dot\varphi_s-J\nabla H(\varphi_s)\bigr)
	\right|^2\,\mathrm ds,
	& \begin{array}{l}
	  \varphi\in H^1([0,T];\mathbb R^{2n}),~~
	  \varphi([0,T])\Subset\mathbb D,
	  \end{array}\\[3ex]
	+\infty, & \text{otherwise}.
\end{cases}
\]

The third main result of this paper is stated as follows.
\begin{theorem}\label{T3}
	Assume that Conditions \((C1)\)--\((C2)\) hold, and let
	\(\mathbb A\subset\mathscr C_{x_0}\) be Borel. Suppose that there exists
	a compact set \(K\Subset\mathbb D\) such that $ \bigcup_{\varphi\in\overline{\mathbb A}}\varphi([0,T])\subset K $, where closure and interior are taken in \(\mathscr C_{x_0}\). Then
	\[
	-\inf_{\varphi\in \mathbb A^\circ}\mathbb I_{\mathbb D}(\varphi)
	\le
	\liminf_{\gamma\to 0}\gamma^2
	\log\mathbb P\bigl(X^\gamma\in\mathbb A,\tau_{\mathbb D}^\gamma>T\bigr)
	\le
	\limsup_{\gamma\to 0}\gamma^2
	\log\mathbb P\bigl(X^\gamma\in\mathbb A,\tau_{\mathbb D}^\gamma>T\bigr)
	\le
	-\inf_{\varphi\in\overline{\mathbb A}}\mathbb I_{\mathbb D}(\varphi).
	\]
\end{theorem}

\textbf{4) We establish the persistence of invariant tori for the stochastic Hamiltonian system \eqref{1} in the sense of most probable paths.}

The fourth result combines the preceding path-probability characterization with classical deterministic KAM theory. In the nearly integrable setting, we prove that invariant tori with Diophantine frequencies persist under small deterministic perturbations and multiplicative noise in the sense of most probable paths, thereby providing a new theoretical perspective on the stability of Hamiltonian structures under random perturbations.

\begin{definition}
	Here a vector $\omega \in \mathbb{R}^d$ is said to be $(\alpha,\tau)$-Diophantine if there exist constants
	$\alpha>0$ and $\tau>0$ such that
	\[
	\left| \omega \cdot k \right|  \ge \frac{\alpha}{\left| k \right|_1^{\tau}},\quad \left| k \right|_1:= \sum_{i=1}^d \left| k_i \right| ,\quad \forall\, k \in \mathbb{Z}^d \setminus \{0\}.
	\]
\end{definition}

The last main result of this paper is stated as follows.
\begin{theorem}\label{T4}
	Consider the stochastic Hamiltonian system \eqref{1} and assume that Conditions \((C1)\)--\((C4)\) hold. Suppose that there is a \(C^l\)-symplectic action--angle chart $ \Psi:\mathbb T^n\times\mathbb U\longrightarrow
	\mathbb D_{\mathrm{aa}}\Subset\mathcal O $ in which \(\mathbb U\subset\mathbb R^n\) is a bounded domain and
	\[
	 H\circ\Psi(\theta,I)=H_0(I)+P(\theta,I).
	\]
	Assume \(H_0\in C^l(\overline{\mathbb U})\) and
	\(P\in C^l(\mathbb T^n\times\overline{\mathbb U})\).
	Then every finite-time orbit of the corresponding deterministic
	Hamiltonian system \eqref{2} contained in \(\mathbb D_{\mathrm{aa}}\) is the
	unique most probable path of \eqref{1} with the same initial point.

	Assume in addition that \(n\ge2\), \(\nu>n\), \(l>2\nu\), and that
	\(H_0\) is Kolmogorov nondegenerate with
	\[
	 \sup_{I\in\mathbb U}
	 \bigl\|(D_I^2H_0(I))^{-1}\bigr\|_2<\infty.
	\]
	If
	\[
	 \eta:=\|P\|_{C^l(\mathbb T^n\times\mathbb U)}
	\]
	is sufficiently small, then there exists a Cantor family of invariant
	Lagrangian tori of \(H_0+P\) corresponding, through a small
	frequency-preserving deformation, to interior invariant tori of \(H_0\)
	with \((\alpha,\tau)\)-Diophantine frequencies, where
	\[
	 \tau=\nu-1,\qquad
	 \alpha\simeq\eta^{1/2-\nu/l}.
	\]
	The deterministic Hamiltonian flow on each surviving torus is conjugate
	to the rigid Diophantine rotation with the same frequency, and every
	finite-time orbit on that torus is the unique most probable path of
	\eqref{1} with the same initial point. No samplewise invariance of the
	stochastic trajectories is asserted.

	Finally, if \(\sigma\) is replaced by \(\gamma\sigma\), then, for every fixed \(T>0\), the local large deviation bounds of Theorem~\ref{T3} hold with speed \(\gamma^{-2}\) and rate function \(\mathbb I_{\mathbb D}\).
\end{theorem}

The main technical difficulty is the state dependence of \(\sigma(X_t)\), which destroys the direct translated-Gaussian representation used for additive noise in \cite{47}. Under (C3), the diffusion can be flattened locally, but the Girsanov density still contains state-dependent first- and second-order terms. The proof therefore requires localization, conditional small-ball estimates, and a careful transfer between the Cameron--Martin measure and the original measure.

Throughout the proof, we repeatedly employ Girsanov transformations to construct new Brownian motions, which allows us to obtain refined estimates for $ {\mathbb{P}\bigl(X(\cdot)\in \mathcal K_U(\varphi,\varepsilon)\bigr)} $, where \(\mathcal K_U(\varphi,\varepsilon)\) is the maximum tube centered at \(\varphi\) in the local coordinates generated by \(U\). The argument uses the exact coordinate-tube representation, It\^o integration by parts, conditional small-ball estimates, and exponential martingale inequalities. Condition (C3) is the structural hypothesis needed for the coordinate flattening in Theorem~\ref{T1}. Condition (C4) is used only afterward: it makes every diffusion column divergence free and reduces the action to its nonnegative quadratic part in Theorem~\ref{T2}.

The remainder of this paper is organized as follows. In Section 2, we review the required function spaces, define the coordinate-tube Onsager--Machlup action and most probable paths, state a finitely differentiable KAM theorem, and collect the technical lemmas. Sections 3--6 prove Theorems~\ref{T1}--\ref{T4}, respectively. Section 6 also gives an explicit state-dependent multiplicative-noise example satisfying (C1)--(C4), which shows that the structural assumptions are nonempty and not merely formal.

\section{Preliminaries}

\subsection{Basic space and norm}
In this section, we review some fundamental definitions and results concerning asymptotic limits in the Wiener space (see reference \cite{16}), as well as several basic definitions of functional spaces and norms (see reference \cite{52}).

Let \(W=(W_t)_{0\le t\le T}\) be a \(2n\)-dimensional Brownian motion on the complete filtered probability space
\[
 (\Omega,\mathcal F,(\mathcal F_t)_{t\ge0},\mathbb P).
\]
Here \(\Omega=C_0([0,T];\mathbb R^{2n})\) is the classical Wiener space and \(\mathbb P\) is Wiener measure. Define
\begin{displaymath}
	\mathbb H_0
	:=\left\{f\in H^1([0,T];\mathbb R^{2n}):f(0)=0\right\}.
\end{displaymath}
This is the Cameron--Martin space associated with \(W\), equipped with
\begin{displaymath}
	\langle f,g\rangle_{\mathbb H_0}
	:=\int_0^T\langle\dot f_s,\dot g_s\rangle\,\mathrm ds,
	\qquad
	\|f\|_{\mathbb H_0}^2
	:=\int_0^T|\dot f_s|^2\,\mathrm ds.
\end{displaymath}

Let \(\mathcal P:\mathbb H_0\to\mathbb H_0\) be an
orthogonal projection with \(N:=\dim(\mathcal P\mathbb H_0)<\infty\).
If \((h_1,\ldots,h_N)\) is an orthonormal basis of
\(\mathcal P\mathbb H_0\), then
\[
 \mathcal Pf=\sum_{i=1}^N\langle f,h_i\rangle_{\mathbb H_0}h_i,
 \qquad
 \mathcal P^W=\sum_{i=1}^N
 \left(\int_0^T\langle\dot h_i(s),\mathrm dW_s\rangle\right)h_i.
\]
The \(\mathbb H_0\)-valued random variable \(\mathcal P^W\) is independent
of the choice of orthonormal basis.
\begin{definition}
	A sequence of orthogonal projections \((\mathcal P_k)_{k\ge1}\) on \(\mathbb H_0\) is called an approximating sequence if \(\dim(\mathcal P_k\mathbb H_0)<\infty\) and \(\mathcal P_k\) converges strongly to \(\operatorname{Id}_{\mathbb H_0}\) as \(k\to\infty\).
\end{definition}

\begin{definition}\label{definition 2.2}
	A seminorm \(\mathcal N\) on \(\mathbb H_0\) is called measurable if there exists an almost surely finite random variable \(\widetilde{\mathcal N}\) such that, for every approximating sequence \((\mathcal P_k)_{k\ge1}\), \(\mathcal N(\mathcal P_k^W)\) converges to \(\widetilde{\mathcal N}\) in probability and \(\mathbb P(\widetilde{\mathcal N}\le\varepsilon)>0\) for every \(\varepsilon>0\). If \(\mathcal N\) is a norm, it is called a measurable norm.
\end{definition}
The same definitions apply to the scalar Cameron--Martin space on \([0,1]\) used in Lemma~\ref{lemma 2.10}.

\begin{definition}
	Let \(l\in \mathbb{N}\). We denote by $ C^{l}\big(\overline{\mathbb{D}},\mathbb{R}\big) $ the space of all real-valued functions $ f:\overline{\mathbb{D}}\to \mathbb{R} $ such that \(f\) admits an extension \(\tilde f\in C^{l}(\mathcal V,\mathbb{R})\) on some open neighborhood
	\(\mathcal V\subset \mathbb{R}^{2n}\) of \(\overline{\mathbb{D}}\), with \(\tilde f|_{\overline{\mathbb{D}}}=f\).
	Equivalently, all partial derivatives of \(f\) up to order \(l\) exist in \(\mathbb{D}\)
	and extend continuously to \(\overline{\mathbb{D}}\). More precisely,
	\[
	f\in C^{l}\big(\overline{\mathbb{D}},\mathbb{R}\big)
	\quad \Longleftrightarrow \quad
	\partial^{\alpha}f \in C\big(\overline{\mathbb{D}}\big)
	\]
	for every multi-index \(\alpha\in \mathbb{N}^{2n}\) satisfying \(|\alpha|\le l\).
\end{definition}

\begin{definition}
	We denote by $ C^{2}\big(\overline{\mathbb{D}},\mathbb{R}^{2n\times 2n}\big) $ the space of all matrix-valued functions $ A:\overline{\mathbb{D}}\to \mathbb{R}^{2n\times 2n} $ such that each entry \(a_{ij}\) belongs to
	\(C^{2}\big(\overline{\mathbb{D}},\mathbb{R}\big)\) for \(1\le i,j\le 2n\). Equivalently,
	\[
	A\in C^{2}\big(\overline{\mathbb{D}},\mathbb{R}^{2n\times 2n}\big)
	\]
	if and only if all first- and second-order partial derivatives of each component
	\(a_{ij}\) exist in \(\mathbb{D}\) and extend continuously to \(\overline{\mathbb{D}}\).
\end{definition}

\begin{definition}
	For a matrix \(A=[a_{ij}]\in\mathbb R^{2n\times 2n}\), its spectral norm is
	\[
	\|A\|_2:=\max_{|x|=1}|Ax|.
	\]
	Equivalently,
	\[
	\|A\|_2=s_{\max}(A),
	\]
	where \(s_{\max}(A)\) is the largest singular value of \(A\), namely, the square root of the largest eigenvalue of \(A^\top A\).
\end{definition}
\begin{definition}
	Let \(A=[a_{ij}]\in\mathbb R^{2n\times 2n}\). Its Hilbert--Schmidt norm is
	\[
	\|A\|_{\mathrm{HS}}
	:=
	\sqrt{
		\sum_{i=1}^{2n} \sum_{j=1}^{2n} |a_{ij}|^2
	}
	=
	\sqrt{\operatorname{tr}\bigl(A^\top A\bigr)}.
	\]
	Equivalently, if \(s_1(A),\ldots,s_{2n}(A)\) are the singular values of \(A\), then
	\[
	\|A\|_{\mathrm{HS}}
	=
	\sqrt{
		\sum_{i=1}^{2n} s_i(A)^2
	}.
	\]
\end{definition}

\begin{definition}
	Let \(f\) be a continuous function on a compact set \(\mathcal D\subset\mathbb R^{2n}\). Its supremum norm is
	\begin{displaymath}
		\|f\|_{0;\mathcal D}:=\sup_{x\in\mathcal D}|f(x)|,
	\end{displaymath}
	where \(|\cdot|\) denotes the Euclidean norm in the target space.
\end{definition}

For \(h=(h^1,\ldots,h^{2n})\in C([0,T];\mathbb R^{2n})\) and \(0<t\le T\), set
\[
 \|h\|_{\infty,*;[0,t]}
 :=\max_{1\le j\le 2n}\sup_{0\le s\le t}|h^j(s)|,
 \qquad
 \|h\|_{\infty,*}:=\|h\|_{\infty,*;[0,T]}.
\]
It is a measurable completely convex norm on the classical Wiener space, so the conditional exponential estimates below apply to its balls.

\subsection{Onsager-Machlup functional}

For a nondegenerate diffusion, every prescribed continuous path has
probability zero; path likelihood is therefore compared through shrinking
tubes. In the local chart from \((C3)\), fix \(x_0\in\mathcal O\), set $ y_0:=U^{-1}(x_0) $, and introduce the fixed-initial, free-terminal admissible class
\[
\mathscr A_{x_0}(\mathcal O)
:=
\left\{
\varphi\in H^1([0,T];\mathbb R^{2n}):
\varphi(0)=x_0,\ 
\varphi([0,T])\Subset\mathcal O
\right\}.
\]
Since \(U^{-1}\) is \(C^2\), every
\(\varphi\in\mathscr A_{x_0}(\mathcal O)\) satisfies
\[
\eta_\varphi
:=
U^{-1}\circ\varphi-y_0
\in\mathbb H_0,
\]
and
\[
\dot{\eta}_\varphi(t)
=
DU^{-1}(\varphi_t)\dot{\varphi}_t
=
\sigma(\varphi_t)^{-1}\dot{\varphi}_t
\qquad\text{for a.e. }t\in[0,T].
\]
The sample path in a tube is required only to be continuous; it is not
required to belong to the Cameron--Martin space. For
\(\varphi\in\mathscr A_{x_0}(\mathcal O)\), the relevant event is the
coordinate tube \(\mathcal K_U(\varphi,\varepsilon)\) defined in
\eqref{local maximum tube}.

\begin{definition}\label{definition OM action}
	A functional $ \mathcal A: \mathscr A_{x_0}(\mathcal O) \longrightarrow\mathbb R $ is called the Onsager--Machlup action relative to the \(U\)-coordinate maximum tubes if
	\[
	\lim_{\varepsilon\to0}
	\frac{
		\mathbb P\bigl(
		X\in\mathcal K_U(\varphi,\varepsilon)
		\bigr)
	}{
		\mathbb P\bigl(
		\|W\|_{\infty,*}\le\varepsilon
		\bigr)
	}
	=
	\exp\left\{
	-\frac12\mathcal A(\varphi)
	\right\}
	\]
	for every
	\(\varphi\in\mathscr A_{x_0}(\mathcal O)\).
	When
	\[
	\mathcal A(\varphi)
	=
	\int_0^T
	OM(t,\varphi_t,\dot{\varphi}_t)\,\mathrm dt,
	\]
	the integrand
	\(OM(t,\varphi,\dot{\varphi})\)
	is called the corresponding Onsager--Machlup function.
\end{definition}

\begin{definition}\label{definition MPP}
	A path
	\(\varphi^*\in\mathscr A_{x_0}(\mathcal O)\)
	is called a most probable path relative to the
	\(U\)-coordinate maximum tubes if, for every
	\(\varphi\in\mathscr A_{x_0}(\mathcal O)\),
	\[
	\limsup_{\varepsilon\to0}
	\frac{
		\mathbb P\bigl(
		X\in\mathcal K_U(\varphi,\varepsilon)
		\bigr)
	}{
		\mathbb P\bigl(
		X\in\mathcal K_U(\varphi^*,\varepsilon)
		\bigr)
	}
	\le 1.
	\]
	It is unique if the above limit superior is strictly smaller than one
	for every $ \varphi\in \mathscr A_{x_0}(\mathcal O)\setminus\{\varphi^*\} $.
\end{definition}

\subsection{KAM Theory}

In Hamiltonian mechanics, an invariant torus is an invariant submanifold of phase space on which the trajectories are quasi-periodic. Such tori may be viewed as higher-dimensional analogues of periodic orbits: when the components of the frequency vector are rationally independent, the motion on the torus is quasi-periodic. KAM theory is concerned with the stability of these invariant tori under small perturbations. For a nearly integrable Hamiltonian system, whose Hamiltonian consists of an integrable part plus a small perturbation, KAM theory asserts that if the perturbation is sufficiently small and appropriate hypotheses hold, then a family of invariant tori of sufficiently large measure persists, up to a small deformation.

We recall the following qualitative formulation of the finitely differentiable KAM theorem. The precise quantitative statement needed below will be invoked directly in the proof of Theorem~\ref{T4}.

\begin{theorem}[\cite{6}]\label{T2.9}
	Consider a Hamiltonian of the form $ H(\theta, I) = H_0(I) + P(\theta, I) $, where $ \theta \in \mathbb{T}^n $ are the angle variables and $ I \in \mathbb{U} \subset \mathbb{R}^n $ are the action variables. Here, $ H_0(I) $ and $ P(\theta, I) $ are $ C^l $-smooth functions with $ H_0, P \in C^l(\mathbb{T}^n) \times \mathbb{U} $, where $ \mathbb{U} $ is a non-empty bounded domain in $ \mathbb{R}^n $. If $H_0$ is non-degenerate and $l > 2\nu > 2n$, then all the KAM tori of the integrable system $H_0$ whose frequency are $(\alpha, \tau)$-Diophantine, with $\alpha \simeq \eta^{1/2 - \nu / l}$ and $\tau := \nu - 1$, do survive, being only slightly deformed, where $\eta$ is the $C^l$-norm of the perturbation $P$. Moreover, letting $\mathcal{K}$ be the corresponding family of KAM tori of $H$, we have 
	\[
	\mathrm{meas}( \mathbb{T}^n \times \mathbb{U} \setminus \mathcal{K}) = O(\eta^{1/2 - \nu / l}).
	\]
\end{theorem}

Theorem~\ref{T2.9} records only the qualitative conclusion and the asymptotic order stated in \cite[Theorem~3]{6}. The Cantor parameter set, the frequency-preserving correspondence, and the quantitative smallness conditions used in this paper are supplied by \cite[Theorem~4, Part~I and Remark~5(i)]{6}.

\subsection{Technical lemmas}
In this section, we introduce the technical lemmas used below. The symbol \(|\cdot|\) denotes the Euclidean norm, \(\|\cdot\|_2\) the matrix operator norm, \(\|\cdot\|_{\mathrm{HS}}\) the Hilbert--Schmidt norm, and \(\|\cdot\|_{0;\mathcal D}\) the supremum norm on a specified domain \(\mathcal D\). Path-space maximum norms are written as \(\|\cdot\|_{\infty,*}\). We use \(C\) for a finite positive constant that may change from line to line. The integer \(n\) is the number of degrees of freedom, so the phase-space dimension is \(2n\), and \(T>0\) is the fixed time horizon. Conditional expectations are written as \(\mathbb E[Z\mid\mathcal G]\) when conditioning on a sigma-algebra and as \(\mathbb E[Z\mid D]\) when conditioning on an event \(D\) of positive probability.

The following elementary product lemma permits the conditional exponential
limits to be combined. For a proof, see \cite[pp.~536--537]{39}.
\begin{lemma}[\cite{39}]\label{lemma 2.9}
Let \(N\ge1\), let \(X_1^\varepsilon,\ldots,X_N^\varepsilon\) be real random variables, and let
\(D_\varepsilon\in\mathcal F\) with \(\mathbb P(D_\varepsilon)>0\). Suppose
that, for every \(c\in\mathbb R\) and every \(i\),
\[
 \limsup_{\varepsilon\to 0}
 \mathbb E[\exp\{cX_i^\varepsilon\}\mid D_\varepsilon]\le1.
\]
Then, for arbitrary \(c_1,\ldots,c_N\in\mathbb R\),
\[
 \lim_{\varepsilon\to 0}
 \mathbb E\left[\exp\left\{\sum_{i=1}^Nc_iX_i^\varepsilon\right\}
 \middle|D_\varepsilon\right]=1.
\]
\end{lemma}

The following two lemmas are fundamental in the calculation of the Onsager--Machlup functional.
\begin{lemma}[\cite{37}]\label{lemma 2.10}
	Let \(\mathcal N\) be a measurable norm on the scalar Cameron--Martin space
	\(\{h\in H^1([0,1];\mathbb R):h(0)=0\}\). For every \(f\in L^2([0,1])\),
	\begin{displaymath}
		\lim_{\varepsilon\to0}\mathbb E\left[
		\exp\left\{\int_0^1 f(s)\,\mathrm dW_s\right\}
		\middle|\widetilde{\mathcal N}\le\varepsilon\right]=1,
	\end{displaymath}
	where \(\widetilde{\mathcal N}\) is the measurable extension in the sense of Definition~\ref{definition 2.2}.
\end{lemma}

\begin{corollary}\label{corollary vector CM maximum ball}
	Let \(B=(B^1,\ldots,B^{2n})\) be a standard \(2n\)-dimensional
	Brownian motion on \([0,T]\). For every deterministic
	\(g\in L^2([0,T];\mathbb R^{2n})\) and every \(c\in\mathbb R\),
	\[
	 \lim_{\varepsilon\to0}
	 \mathbb E\left[
	 \exp\left\{c\int_0^T\langle g_s,\mathrm dB_s\rangle\right\}
	 \middle|\|B\|_{\infty,*}\le\varepsilon\right]=1.
	\]
\end{corollary}

\begin{proof}
	On the scalar Cameron--Martin space on \([0,1]\), the supremum
	norm is measurable, and its measurable extension in the sense of
	Definition~\ref{definition 2.2} is \(\sup_{0\le u\le1}|B_u^k|\).
	For \(0\le u\le1\), set
	\(\overline B_u^k:=T^{-1/2}B_{Tu}^k\) and
	\(g_{k,T}(u):=\sqrt T\,g_k(Tu)\). Then
	\[
	 \int_0^Tg_k(s)\,\mathrm dB_s^k
	 =\int_0^1g_{k,T}(u)\,\mathrm d\overline B_u^k,
	\]
	while
	\[
	 \{\|B\|_{\infty,*}\le\varepsilon\}
	 =\bigcap_{k=1}^{2n}
	 \left\{\|\overline B^k\|_\infty
	 \le\frac{\varepsilon}{\sqrt T}\right\}.
	\]
	The scalar Brownian motions \(\overline B^k\) are independent. Hence the conditional expectation factorizes over \(k\), and each factor converges to one by Lemma~\ref{lemma 2.10}, applied to the scalar supremum norm and to \(cg_{k,T}\in L^2([0,1])\). The finite product therefore converges to one.
\end{proof}

\begin{lemma}[\cite{60}]\label{lemma 2.15}
	Let \(W=(W^1,\ldots,W^d)\) be a \(d\)-dimensional Brownian motion, let \(\|\cdot\|\) be a completely convex norm on its path space, and fix \(i\in\{1,\ldots,d\}\). Let \(\mathcal F_i\) be the sigma-algebra generated by \(\{W_s^j:j\ne i,\ 0\le s\le T\}\). Suppose that \(\Psi\) is an \(\mathcal F_i\)-measurable process such that, for every \(c\in\mathbb R\),
	\begin{displaymath}
		\lim_{\varepsilon\to0}\mathbb E\left[
		\exp\left\{c^2\int_0^T|\Psi_s|^2\,\mathrm ds\right\}
		\middle|\|W\|<\varepsilon\right]=1.
	\end{displaymath}
	Then, for every \(c\in\mathbb R\),
	\begin{displaymath}
		\lim_{\varepsilon\to0}\mathbb E\left[
		\exp\left\{c\int_0^T\Psi_s\,\mathrm dW_s^i\right\}
		\middle|\|W\|<\varepsilon\right]=1.
	\end{displaymath}
\end{lemma}

\begin{corollary}\label{corollary off diagonal maximum ball}
	Let \(B=(B^1,\ldots,B^{2n})\) be a standard \(2n\)-dimensional
	Brownian motion on \([0,T]\), and let
	\(F:[0,T]\to\mathbb R^{2n\times2n}\) be deterministic, Borel measurable,
	and bounded.
	Then, for every \(i\ne j\) and every \(c\in\mathbb R\),
	\[
	 \lim_{\varepsilon\to0}
	 \mathbb E\left[
	 \exp\left\{c\int_0^T F_s^{ij}B_s^j\,\mathrm dB_s^i\right\}
	 \middle|\|B\|_{\infty,*}\le\varepsilon\right]=1.
	\]
\end{corollary}

\begin{proof}
	Fix \(i\ne j\). Set
	\[
	 \mathcal H_i:=\sigma\{B_s^k:k\ne i,\ 0\le s\le T\}
	\]
	and, after the usual completion,
	\[
	 \mathcal G_t^{(i)}:=\mathcal H_i\vee
	 \sigma\{B_s^i:0\le s\le t\},
	 \qquad 0\le t\le T.
	\]
	Independence of the Brownian coordinates
	implies that \(B^i\) remains a Brownian motion relative to
	\((\mathcal G_t^{(i)})_{0\le t\le T}\). The process
	\(F^{ij}B^j\) is \((\mathcal G_t^{(i)})\)-predictable and is measurable
	with respect to \(\mathcal H_i\) as an \(L^2([0,T])\)-valued random
	element.

	We use the symmetric-convex Gaussian estimate underlying
	Lemma~\ref{lemma 2.15}. If \(f\in L^2([0,T])\) is deterministic, then
	for every \(c\in\mathbb R\),
	\[
	 1\le
	 \mathbb E\left[
	 \exp\left\{c\int_0^Tf_s\,\mathrm dB_s^i\right\}
	 \middle|\|B^i\|_\infty\le\varepsilon\right]
	 \le
	 \exp\left\{\frac{c^2}{2}\int_0^T|f_s|^2\,\mathrm ds\right\}.
	\]
	Indeed, the conditional law on the symmetric convex ball
	\(\{\|B^i\|_\infty\le\varepsilon\}\) is symmetric, which gives the lower bound by
	\(\cosh z\ge1\); the convex Gaussian estimate gives the upper bound.
	Conditioning first on \(\mathcal H_i\), applying this estimate with
	\(f_s=F_s^{ij}B_s^j\), and then restricting the remaining coordinates to
	\(\bigcap_{k\ne i}\{\|B^k\|_\infty\le\varepsilon\}\), we obtain
	\[
	 1\le
	 \mathbb E\left[
	 \exp\left\{c\int_0^TF_s^{ij}B_s^j\,\mathrm dB_s^i\right\}
	 \middle|\|B\|_{\infty,*}\le\varepsilon\right]
	 \le
	 \exp\left\{\frac{c^2T}{2}
	 \|F^{ij}\|_{0;[0,T]}^2\varepsilon^2\right\}.
	\]
	The conclusion follows by squeezing as \(\varepsilon\to0\).
\end{proof}

\section{Proof of Theorem \ref{T1}}

\begin{proof}[Proof of Theorem \ref{T1}]
	We first make the localization and the change of measure precise. Let
	\(\psi:=U^{-1}\circ\varphi\) and choose bounded open sets with smooth
	boundary such that
	\[
	 \psi([0,T])\Subset\widetilde{\mathcal O}_0
	 \Subset\widetilde{\mathcal O}_1
	 \Subset\widetilde{\mathcal O}.
	\]
	Put \(\mathcal O_i:=U(\widetilde{\mathcal O}_i)\), \(i=0,1\), and choose
	\[
	 0<\varepsilon_0<
	 \min_{i=0,1}\operatorname{dist}_{\infty,*}
	 (\psi([0,T]),\partial\widetilde{\mathcal O}_i).
	\]
	Extend \(J\nabla H\) and \(\sigma\) to globally bounded Lipschitz
	coefficients on \(\mathbb R^{2n}\) that agree with the original
	coefficients near \(\overline{\mathcal O}_1\). Pathwise uniqueness and the
	localization below identify the resulting global solution with the original
	local solution on every tube under consideration.
	Thus, whenever \(0<\varepsilon<\varepsilon_0\), both a path in
	\(\mathcal K_U(\varphi,\varepsilon)\) and its coordinate representative
	remain in \(\mathcal O_0\) and \(\widetilde{\mathcal O}_0\), respectively.
	Let \(y_0:=U^{-1}(x_0)\) and define
	\[
	 \zeta:=\inf\{t\ge0:y_0+W_t\notin\widetilde{\mathcal O}_1\}.
	\]
	The corresponding exit time for the original process is
	\[
	 \tau_1:=\inf\{t\ge0:X_t\notin\mathcal O_1\}.
	\]
	The stopped auxiliary process is
	\begin{equation}\label{4}
	 Y_t:=U(y_0+W_{t\wedge\zeta}).
	\end{equation}
	Before \(\zeta\), the Stratonovich chain rule and (C3) give
	\[
	 \mathrm dY_t=DU(y_0+W_t)\circ\mathrm dW_t
	 =\sigma(Y_t)\circ\mathrm dW_t,
	 \qquad Y_0=x_0.
	\]
	Thus the flattening map provides an exact Wiener-coordinate
	representation up to the exit time on which the local chart is valid.

	For comparison, the It\^o forms before the corresponding exit times are
	\begin{equation*}
	 \mathrm dX_t=
	 \left[J\nabla H(X_t)+\frac12\sum_{\alpha=1}^{2n}
	 D\sigma_\alpha(X_t)\sigma_\alpha(X_t)\right]\mathrm dt
	 +\sigma(X_t)\mathrm dW_t
	\end{equation*}
	and
	\begin{equation}\label{6}
	 \mathrm dY_t=
	 \frac12\sum_{\alpha=1}^{2n}
	 D\sigma_\alpha(Y_t)\sigma_\alpha(Y_t)\,\mathrm dt
	 +\sigma(Y_t)\mathrm dW_t.
	\end{equation}
	Since \(\overline{\mathcal O}_1\Subset\mathcal O\), the map
	\(x\mapsto\sigma(x)^{-1}J\nabla H(x)\) is bounded on
	\(\overline{\mathcal O}_1\). Define the stopped exponential
	\begin{equation}\label{stopped Girsanov density}
	 \mathcal R:=\exp\left\{
	 \int_0^{T\wedge\zeta}
	 \left\langle\sigma(Y_s)^{-1}J\nabla H(Y_s),\mathrm dW_s\right\rangle
	 -\frac12\int_0^{T\wedge\zeta}
	 \left|\sigma(Y_s)^{-1}J\nabla H(Y_s)\right|^2\,\mathrm ds\right\}.
	\end{equation}
	Set
	\[
	 \frac{\mathrm d\mathbb Q}{\mathrm d\mathbb P}=\mathcal R.
	\]
	The preceding boundedness implies Novikov's condition and hence
	\(\mathbb E_{\mathbb P}\mathcal R=1\). Girsanov's theorem shows that
	\begin{equation*}
	 \widetilde W_t=W_t-
	 \int_0^{t\wedge\zeta}
	 \sigma(Y_s)^{-1}J\nabla H(Y_s)\,\mathrm ds
	\end{equation*}
	is Brownian under \(\mathbb Q\). Substitution in \eqref{6} gives, for
	\(t<\zeta\),
	\begin{equation*}
	 \mathrm dY_t=
	 \left[J\nabla H(Y_t)+\frac12\sum_{\alpha=1}^{2n}
	 D\sigma_\alpha(Y_t)\sigma_\alpha(Y_t)\right]\mathrm dt
	 +\sigma(Y_t)\mathrm d\widetilde W_t.
	\end{equation*}
	Local Lipschitz continuity yields weak uniqueness for the stopped equation. Therefore
	\[
	 \mathcal L_{\mathbb P}(X_{\cdot\wedge\tau_1})
	 =\mathcal L_{\mathbb Q}(Y_{\cdot\wedge\zeta}).
	\]
	On either coordinate tube with \(\varepsilon<\varepsilon_0\), the
	corresponding exit time exceeds \(T\). Consequently the preceding equality
	of stopped laws gives an exact equality of the tube probabilities. Moreover,
	on that event the stopped density \(\mathcal R\) equals the same exponential
	in \eqref{stopped Girsanov density} with both integrals taken over
	\([0,T]\). All subsequent identities may therefore be written with the
	original coefficients.

	Define
	\[
	 \eta_t:=\psi_t-y_0=U^{-1}(\varphi_t)-y_0,
	 \qquad \eta_0=0.
	\]
	The definition of the tube and \eqref{4} give the exact event identity
	\begin{equation}\label{9}
	 \{Y\in\mathcal K_U(\varphi,\varepsilon)\}
	 =\{\|W-\eta\|_{\infty,*}\le\varepsilon\}
	 =:A_\varepsilon.
	\end{equation}
	It follows that
	\begin{equation}\label{tube probability factorization}
	 \begin{aligned}
	 \mathbb P\bigl(X\in\mathcal K_U(\varphi,\varepsilon)\bigr)
	 &=\mathbb Q\bigl(Y\in\mathcal K_U(\varphi,\varepsilon)\bigr)\\
	 &=\mathbb E_{\mathbb P}\bigl[\mathcal R\mathbf 1_{A_\varepsilon}\bigr]\\
	 &=\mathbb E_{\mathbb P}[\mathcal R\mid A_\varepsilon]\,
	 \mathbb P(A_\varepsilon).
	 \end{aligned}
	\end{equation}
	We complete the proof in three steps.

	\textbf{First step.}
	We estimate the conditional expectation of \(\mathcal R\). Put \(h_t:=Y_t-\varphi_t\). On \(A_\varepsilon\), the mean-value theorem and boundedness of \(DU\) on the compact localization set give \(\|h\|_{\infty,*}\le C\varepsilon\). Moreover, Conditions \((C1)\)--\((C2)\) imply that the map \(x\mapsto\sigma(x)^{-1}J\nabla H(x)\) is \(C^2\) there. Taylor's formula therefore gives
	\begin{equation}\label{10}
	 \begin{aligned}
	 \sigma(Y_t)^{-1}J\nabla H(Y_t)
	 =\sigma(\varphi_t)^{-1}J\nabla H(\varphi_t)
	 +D\!\left[\sigma(\cdot)^{-1}J\nabla H(\cdot)\right](\varphi_t)h_t
	 +R_t^{(2)}.
	 \end{aligned}
	\end{equation}
	The second-order remainder satisfies
	\[
	 \|R^{(2)}\|_{\infty,*}\le C\varepsilon^2.
	\]
	Indeed, with \(\hbar:=W-\eta\),
	\begin{equation}\label{12}
	 \begin{aligned}
	 h_t
	 &=U(y_0+\eta_t+\hbar_t)-U(y_0+\eta_t)\\
	 &=\sigma(\varphi_t)\hbar_t+R_{U,t}^{(2)}.
	 \end{aligned}
	\end{equation}
	The corresponding remainder satisfies
	\[
	 \|R_U^{(2)}\|_{\infty,*}\le C\varepsilon^2.
	\]
	Accordingly, the exponent in \(\mathcal R\) is decomposed as
	\[
	 \begin{aligned}
	 B
	 &:=\int_0^T\left\langle
	 \sigma(Y_s)^{-1}J\nabla H(Y_s),\mathrm dW_s\right\rangle
	 -\frac12\int_0^T
	 \left|\sigma(Y_s)^{-1}J\nabla H(Y_s)\right|^2\,\mathrm ds\\
	 &=\int_0^T\left\langle
	 \sigma(\varphi_s)^{-1}J\nabla H(\varphi_s),\mathrm dW_s\right\rangle
	 +\int_0^T\left\langle
	 D\!\left[\sigma(\cdot)^{-1}J\nabla H(\cdot)\right](\varphi_s)h_s,
	 \mathrm dW_s\right\rangle\\
	 &\quad+\int_0^T\langle R_s^{(2)},\mathrm dW_s\rangle
	 -\frac12\int_0^T
	 \left|\sigma(Y_s)^{-1}J\nabla H(Y_s)\right|^2\,\mathrm ds\\
	 &=: B_1 + B_2 + B_3 + B_4.
	 \end{aligned}
	\]

	1) Estimate of \(B_1\).
	Since \(\varphi\) is absolutely continuous and the map \(x\mapsto\sigma(x)^{-1}J\nabla H(x)\) is \(C^1\) on the compact localization set, the composition \(s\mapsto\sigma(\varphi_s)^{-1}J\nabla H(\varphi_s)\) is absolutely continuous and
	\[
	 \frac{\mathrm d}{\mathrm ds}
	 \bigl(\sigma(\varphi_s)^{-1}J\nabla H(\varphi_s)\bigr)
	 =D\!\left[\sigma(\cdot)^{-1}J\nabla H(\cdot)\right](\varphi_s)
	 \dot\varphi_s
	 \quad\text{for a.e. }s.
	\]
	Writing \(W=\eta+\hbar\), stochastic integration by parts gives, component
	by component,
	\begin{align*}
	 \int_0^T\left\langle
	 \sigma(\varphi_s)^{-1}J\nabla H(\varphi_s),\mathrm dW_s\right\rangle
	 &=\int_0^T\left\langle
	 \sigma(\varphi_s)^{-1}J\nabla H(\varphi_s),\mathrm d\eta_s\right\rangle
	 +\left\langle
	 \sigma(\varphi_T)^{-1}J\nabla H(\varphi_T),\hbar_T\right\rangle\\
	 &\quad-\int_0^T\left\langle\hbar_s,
	 D\!\left[\sigma(\cdot)^{-1}J\nabla H(\cdot)\right](\varphi_s)
	 \dot\varphi_s\right\rangle\,\mathrm ds.
	\end{align*}
	Because \(\|\hbar\|_{\infty,*}\le\varepsilon\) on \(A_\varepsilon\), Cauchy--Schwarz yields
	\[
	 \begin{aligned}
	 \left|\left\langle
	 \sigma(\varphi_T)^{-1}J\nabla H(\varphi_T),\hbar_T\right\rangle
	 -\int_0^T\left\langle\hbar_s,
	 D\!\left[\sigma(\cdot)^{-1}J\nabla H(\cdot)\right](\varphi_s)
	 \dot\varphi_s\right\rangle\,\mathrm ds\right| \le C\varepsilon\left(1+\int_0^T|\dot\varphi_s|\,\mathrm ds\right)
	 \le C_\varphi\varepsilon.
	 \end{aligned}
	\]
	Moreover, the chain rule and \((C3)\) give
	\[
	 \dot\eta_s=DU^{-1}(\varphi_s)\dot\varphi_s
	 =\sigma(\varphi_s)^{-1}\dot\varphi_s
	 \quad\text{for a.e. }s.
	\]
	Consequently,
	\[
	 \int_0^T\left\langle
	 \sigma(\varphi_s)^{-1}J\nabla H(\varphi_s),
	 \mathrm d\eta_s\right\rangle
	 =\int_0^T\left\langle
	 \sigma(\varphi_s)^{-1}J\nabla H(\varphi_s),
	 \sigma(\varphi_s)^{-1}\dot\varphi_s\right\rangle\,\mathrm ds,
	\]
	and, uniformly on \(A_\varepsilon\),
	\begin{equation*}
	 B_1=\int_0^T
	 \left\langle\sigma(\varphi_s)^{-1}J\nabla H(\varphi_s),
	 \sigma(\varphi_s)^{-1}\dot\varphi_s\right\rangle\,\mathrm ds
	 +O(\varepsilon).
	\end{equation*}
	Hence, for every \(c\in\mathbb R\),
	\begin{equation}\label{conditional B1 limit}
	 \begin{aligned}
	 &\lim_{\varepsilon\to0}\mathbb E\left[
	 \exp\left\{c\left(B_1-\int_0^T
	 \left\langle\sigma(\varphi_s)^{-1}J\nabla H(\varphi_s),
	 \sigma(\varphi_s)^{-1}\dot\varphi_s\right\rangle
	 \,\mathrm ds\right)\right\}\middle|A_\varepsilon\right]=1.
	 \end{aligned}
	\end{equation}

	2) Estimate of \(B_2\).
	Substitution of \eqref{12} and the identity
	\(\mathrm dW=\mathrm d\hbar+\mathrm d\eta\) give
	\[
	 \begin{aligned}
	 B_2
	 &=\int_0^T\langle F_s\hbar_s,\mathrm d\hbar_s\rangle
	 +\int_0^T\langle F_s\hbar_s,\mathrm d\eta_s\rangle
	 +\int_0^T\left\langle
	 D\!\left[\sigma(\cdot)^{-1}J\nabla H(\cdot)\right](\varphi_s)
	 R_{U,s}^{(2)},\mathrm dW_s\right\rangle\\
	 &:=B_{21}+B_{22}+B_{23},
	 \end{aligned}
	\]
	where
	\[
	 F_s:=D\!\left[\sigma(\cdot)^{-1}J\nabla H(\cdot)\right](\varphi_s)
	 \sigma(\varphi_s).
	\]
	Under \(\mathbb P\), \(\mathrm d\hbar=\mathrm dW-\dot\eta\,\mathrm dt\).
	Introduce \(\widehat{\mathbb P}\) by the Cameron--Martin change of measure
	under which \(\hbar\) is a standard \(2n\)-dimensional Brownian motion. Then
	\[
	 Z_\eta:=\frac{\mathrm d\mathbb P}{\mathrm d\widehat{\mathbb P}}
	 =\exp\left\{-\int_0^T\langle\dot\eta_s,\mathrm d\hbar_s\rangle
	 -\frac12\|\eta\|_{\mathbb H_0}^2\right\}.
	\]
	Here
	\[
	 \|\eta\|_{\mathbb H_0}^2=\int_0^T|\dot\eta_s|^2\,\mathrm ds.
	\]
	Writing \(D_\varepsilon:=\{\|\hbar\|_{\infty,*}\le\varepsilon\}\), we have
	\[
	 \mathbb E_{\mathbb P}[\exp\{cB_{21}\}\mid A_\varepsilon]
	 =\frac{\mathbb E_{\widehat{\mathbb P}}
	 [\exp\{cB_{21}\}Z_\eta\mathbf1_{D_\varepsilon}]}
	 {\mathbb E_{\widehat{\mathbb P}}
	 [Z_\eta\mathbf1_{D_\varepsilon}]},\qquad c\in\mathbb R.
	\]
	We compute the centered conditional exponential
	moments of \(B_{21}\) under \(\widehat{\mathbb P}\) and then transfer
	them to \(\mathbb P\).

	Decompose \(B_{21}\) into its off-diagonal and
	diagonal parts. Write \(F=(F^{ij})_{1\le i,j\le 2n}\). Then
	\[
	\begin{aligned}
		 B_{21}
		&=\sum_{\substack{1\le i,j\le 2n\\i\ne j}}
		\int_0^T F_s^{ij}\hbar_s^j\,\mathrm d\hbar_s^i
		+\sum_{i=1}^{2n}\int_0^T F_s^{ii}\hbar_s^i\,\mathrm d\hbar_s^i\\
		&=:B_{211}+B_{212}.
	\end{aligned}
	\]
	For \(i\ne j\), the integrand \(F^{ij}\hbar^j\) is predictable and
	measurable with
	respect to the Brownian coordinates other than \(\hbar^i\). On the maximum
	ball,
	\[
	 \left\langle\int_0^\cdot
	 F_s^{ij}\hbar_s^j\,\mathrm d\hbar_s^i\right\rangle_T
	 =\int_0^T|F_s^{ij}\hbar_s^j|^2\,\mathrm ds
	 \le \|F^{ij}\|_{0;[0,T]}^2T\varepsilon^2.
	\]
	In particular, for each \(c\in\mathbb R\),
	\[
	 \mathbb E_{\widehat{\mathbb P}}\left[
	 \exp\left\{c^2\int_0^T|F_s^{ij}\hbar_s^j|^2\,\mathrm ds\right\}
	 \middle|D_\varepsilon\right]
	 \le\exp\{c^2\|F^{ij}\|_{0;[0,T]}^2T\varepsilon^2\}\longrightarrow1.
	\]
	Corollary~\ref{corollary off diagonal maximum ball}, followed by
	Lemma~\ref{lemma 2.9} over the
	finitely many ordered pairs, therefore yields, for every \(c\in\mathbb R\),
	\[
	 \lim_{\varepsilon\to 0}
	 \mathbb E_{\widehat{\mathbb P}}
	 [\exp\{cB_{211}\}\mid\|\hbar\|_{\infty,*}\le\varepsilon]=1.
	\]
	For the diagonal terms, \(F^{ii}\) is absolutely continuous because
	\(\varphi\in H^1([0,T];\mathbb R^{2n})\) and
	\(x\mapsto D[\sigma(\cdot)^{-1}J\nabla H(\cdot)](x)\sigma(x)\) is \(C^1\) on the
	localization set. It\^o's product formula gives
	\[
	 \mathrm d\!\left(F_t^{ii}(\hbar_t^i)^2\right)
	 =(\hbar_t^i)^2\,\mathrm dF_t^{ii}
	 +2F_t^{ii}\hbar_t^i\,\mathrm d\hbar_t^i
	 +F_t^{ii}\,\mathrm dt.
	\]
	Since \(\hbar_0=0\), integration and rearrangement yield
	\begin{equation}\label{diagonal identity}
	 \int_0^T F_s^{ii}\hbar_s^i\,\mathrm d\hbar_s^i
	 =\frac12F_T^{ii}(\hbar_T^i)^2
	 -\frac12\int_0^T(\hbar_s^i)^2\,\mathrm dF_s^{ii}
	 -\frac12\int_0^T F_s^{ii}\,\mathrm ds.
	\end{equation}
	On \(D_\varepsilon\),
	\[
	 \left|F_T^{ii}(\hbar_T^i)^2\right|
	 +\left|\int_0^T(\hbar_s^i)^2\,\mathrm dF_s^{ii}\right|
	 \le\varepsilon^2\bigl(\|F^{ii}\|_{0;[0,T]}+
	 \operatorname{Var}_{[0,T]}F^{ii}\bigr).
	\]
	The conditional exponential moments of these two terms therefore converge
	to one. Summing \eqref{diagonal identity} over \(i\) shows that the only
	nonvanishing contribution of \(B_{212}\) is
	\(-\frac12\int_0^T\operatorname{tr}F_s\,\mathrm ds\).
	Consequently, for every \(c\in\mathbb R\),
	\[
	 \lim_{\varepsilon\to 0}
	 \mathbb E_{\widehat{\mathbb P}}
	 [\exp\{cB_{212}\}\mid D_\varepsilon]
	 =\exp\left\{-\frac c2\int_0^T\operatorname{tr}F_s\,\mathrm ds\right\}.
	\]
	Equivalently, the centered diagonal term
	\(B_{212}+\frac12\int_0^T\operatorname{tr}F_s\,\mathrm ds\)
	has conditional exponential moments converging to one. Combining this
	centered term with the off-diagonal limit by Lemma~\ref{lemma 2.9} yields
	\[
	 \lim_{\varepsilon\to 0}
	 \mathbb E_{\widehat{\mathbb P}}
	 [\exp\{cB_{21}\}\mid D_\varepsilon]
	 =\exp\left\{-\frac c2\int_0^T\operatorname{tr}F_s\,\mathrm ds\right\}.
	\]

	For every \(c\in\mathbb R\), the preceding calculation gives
	\[
	 \mathbb E_{\widehat{\mathbb P}}\left[
	 \exp\left\{c\left(B_{21}
	 +\frac12\int_0^T\operatorname{tr}F_s\,\mathrm ds\right)\right\}
	 \middle|D_\varepsilon\right]\longrightarrow1.
	\]
	Corollary~\ref{corollary vector CM maximum ball} also gives
	\[
	 \mathbb E_{\widehat{\mathbb P}}\left[
	 \exp\left\{-c\int_0^T
	 \langle\dot\eta_s,\mathrm d\hbar_s\rangle\right\}
	 \middle|D_\varepsilon\right]\longrightarrow1.
	\]
	Applying Lemma~\ref{lemma 2.9} to these two centered random variables yields
	the joint limit
	\[
	 \begin{aligned}
	 \mathbb E_{\widehat{\mathbb P}}\left[
	 \exp\left\{c\left(B_{21}
	 +\frac12\int_0^T\operatorname{tr}F_s\,\mathrm ds\right)
	 -\int_0^T\langle\dot\eta_s,\mathrm d\hbar_s\rangle\right\}
	 \middle|D_\varepsilon\right]\longrightarrow1.
	 \end{aligned}
	\]
	The deterministic factor
	\(\exp\{-\|\eta\|_{\mathbb H_0}^2/2\}\) cancels in the conditional
	numerator and denominator. Consequently the exact density-transfer identity
	above gives
	\begin{equation}\label{13}
	 \lim_{\varepsilon\to 0}
	 \mathbb E_{\mathbb P}[\exp\{cB_{21}\}\mid A_\varepsilon]
	 =\exp\left\{-\frac c2\int_0^T\operatorname{tr}F_s\,\mathrm ds\right\},
	 \qquad c\in\mathbb R.
	\end{equation}

	The trace is identified componentwise. For every \(x\) in the localization set,
	\(J\nabla H(x)=\sigma(x)[\sigma(x)^{-1}J\nabla H(x)]\), and hence
	\begin{equation}\label{14}
	 \begin{aligned}
	 &\quad \operatorname{tr}\!\left(
	 D[\sigma(\cdot)^{-1}J\nabla H(\cdot)](x)\sigma(x)\right)\\
	 &=\sum_{\alpha=1}^{2n}\sum_{j=1}^{2n}
	 \partial_j[\sigma(x)^{-1}J\nabla H(x)]_\alpha\,
	 \sigma_{j\alpha}(x)\\
	 &=\sum_{j=1}^{2n}\partial_j
	 \left(\sum_{\alpha=1}^{2n}\sigma_{j\alpha}(x)
	 [\sigma(x)^{-1}J\nabla H(x)]_\alpha\right)
	 -\sum_{\alpha=1}^{2n}\sum_{j=1}^{2n}
	 (\partial_j\sigma_{j\alpha})(x)
	 [\sigma(x)^{-1}J\nabla H(x)]_\alpha\\
	 &=\operatorname{div}(J\nabla H)(x)
	 -(\operatorname{div}\sigma)(x)\cdot
	 [\sigma(x)^{-1}J\nabla H(x)].
	 \end{aligned}
	\end{equation}
	Because \(D^2H\) is symmetric and \(J\) is skew-symmetric,
	\(\operatorname{div}(J\nabla H)=\operatorname{tr}(JD^2H)=0\).
	Combining \eqref{13}--\eqref{14} with \(c=1\), we obtain
	\begin{equation*}
	 \lim_{\varepsilon\to 0}
	 \mathbb E[\exp\{B_{21}\}\mid A_\varepsilon]
	 =\exp\left\{\frac12\int_0^T
	 (\operatorname{div}\sigma)(\varphi_s)\cdot
	 \bigl(\sigma(\varphi_s)^{-1}J\nabla H(\varphi_s)\bigr)
	 \,\mathrm ds\right\}.
	\end{equation*}

	Since \(\eta\in\mathbb H_0\), the second term satisfies the pathwise bound
	\begin{equation}\label{16}
	 |B_{22}|\le
	 \left(\sup_{0\le s\le T}\|F_s\|_2\right)\|\hbar\|_{\infty,*}
	 \int_0^T|\dot\eta_s|\,\mathrm ds\le C\varepsilon
	 \quad\text{on }A_\varepsilon.
	\end{equation}

	For the last component \(B_{23}\), let
	\[
	 \tau_\varepsilon:=\inf\{t\in[0,T]:
	 \|\hbar\|_{\infty,*;[0,t]}>\varepsilon\}\wedge T.
	\]
	The Taylor estimate \eqref{12} gives
	\[
	 \left|D\!\left[\sigma(\cdot)^{-1}J\nabla H(\cdot)\right](\varphi_s)
	 R_{U,s}^{(2)}\right|\le C\varepsilon^2,
	 \qquad 0\le s\le\tau_\varepsilon.
	\]
	Define the stopped martingale
	\[
	 M_t^{(23)}:=\int_0^{t\wedge\tau_\varepsilon}
	 \left\langle
	 D\!\left[\sigma(\cdot)^{-1}J\nabla H(\cdot)\right](\varphi_s)
	 R_{U,s}^{(2)},
	 \mathrm dW_s\right\rangle.
	\]
	Its progressively measurable integrand is bounded by
	\(C\varepsilon^2\), and hence
	\[
	 \langle M^{(23)}\rangle_T\le CT\varepsilon^4.
	\]
	The exponential martingale inequality consequently gives, for \(r>0\),
	\begin{equation}\label{17}
	 \mathbb P(|M_T^{(23)}|>r)
	 \le2\exp\left\{-\frac{r^2}{2CT\varepsilon^4}\right\}.
	\end{equation}
	For every \(c\in\mathbb R\), it also gives
	\begin{equation*}
	 \mathbb E\bigl[\exp\{cM_T^{(23)}\}\bigr]\le
	 \exp\left\{\frac{c^2CT\varepsilon^4}{2}\right\}.
	\end{equation*}
	The Cameron--Martin formula yields
	\[
	 \frac{\mathbb P(A_\varepsilon)}
	 {\mathbb P(\|W\|_{\infty,*}\le\varepsilon)}
	 =\exp\left\{-\frac12\|\eta\|_{\mathbb H_0}^2\right\}
	 \mathbb E_{\widehat{\mathbb P}}\left[
	 \exp\left\{-\int_0^T
	 \langle\dot\eta_s,\mathrm d\hbar_s\rangle\right\}
	 \middle|D_\varepsilon\right].
	\]
	Corollary~\ref{corollary vector CM maximum ball}, independence of the
	Brownian coordinates, and the one-dimensional maximum-ball estimate imply
	\begin{equation}\label{18}
	 \mathbb P(A_\varepsilon)
	 =\mathbb P(\|W-\eta\|_{\infty,*}\le\varepsilon)
	 \ge c_0\exp\{-C_0\varepsilon^{-2}\}
	\end{equation}
	for all sufficiently small \(\varepsilon\). Indeed,
	\[
	 \mathbb P(\|W\|_{\infty,*}\le\varepsilon)
	 =\left[\mathbb P\left(
	 \sup_{0\le t\le T}|B_t|\le\varepsilon\right)\right]^{2n}
	 \ge c\exp\{-C\varepsilon^{-2}\},
	\]
	and the conditional Cameron--Martin factor is bounded below by a positive
	constant for all small \(\varepsilon\).

	On \(\{|M_T^{(23)}|\le r\}\), we have
	\[
	\exp\{-|c|r\}\le\exp\{cM_T^{(23)}\}\le\exp\{|c|r\}.
	\]
	On the complementary event, Cauchy--Schwarz and
	\eqref{17}--\eqref{18} give
	\[
	\frac{
		\mathbb E\!\left[
		\exp\{cM_T^{(23)}\}
		\mathbf 1_{\{|M_T^{(23)}|>r\}\cap A_\varepsilon}
		\right]
	}{
		\mathbb P(A_\varepsilon)
	}
	\le
	C\exp\left\{
	C_0\varepsilon^{-2}
	-\frac{r^2}{8CT\varepsilon^4}
	+C_c\varepsilon^4
	\right\}.
	\]
	With \(r=\varepsilon^{1/2}\), the right-hand side tends to zero; the
	same estimate controls
	\(\mathbb P(|M_T^{(23)}|>r\mid A_\varepsilon)\). Therefore
	\begin{equation}\label{19}
	 \mathbb E[\exp\{cM_T^{(23)}\}\mid A_\varepsilon]
	 \longrightarrow1,\qquad c\in\mathbb R.
	\end{equation}
	Since \(A_\varepsilon\subset\{\tau_\varepsilon=T\}\), one has
	\(M_T^{(23)}=B_{23}\) on \(A_\varepsilon\), and thus
	\begin{equation}\label{20}
	 \mathbb E[\exp\{cB_{23}\}\mid A_\varepsilon]\longrightarrow1,
	 \qquad c\in\mathbb R.
	\end{equation}
	By \eqref{13}--\eqref{14}, \eqref{16}, \eqref{20}, and
	Lemma~\ref{lemma 2.9}, the complete second term satisfies
	\begin{equation}\label{conditional B2 limit}
	 \begin{aligned}
	 &\lim_{\varepsilon\to0}\mathbb E\left[
	 \exp\left\{c\left(B_2-\frac12\int_0^T
	 (\operatorname{div}\sigma)(\varphi_s)\cdot
	 \bigl(\sigma(\varphi_s)^{-1}J\nabla H(\varphi_s)\bigr)
	 \,\mathrm ds\right)\right\}\middle|A_\varepsilon\right]=1,
	 \qquad c\in\mathbb R.
	 \end{aligned}
	\end{equation}

	3) Estimate of \(B_3\).
	The Taylor estimate \eqref{10} gives
	\[
	 |R_s^{(2)}|\le C\varepsilon^2,
	 \qquad 0\le s\le\tau_\varepsilon.
	\]
	Thus
	\[
	 M_t^{(3)}:=\int_0^{t\wedge\tau_\varepsilon}
	 \langle R_s^{(2)},\mathrm dW_s\rangle
	\]
	has quadratic variation bounded by \(CT\varepsilon^4\). Replacing
	\(M^{(23)}\) by \(M^{(3)}\) in \eqref{17}--\eqref{19}, and using
	\(M_T^{(3)}=B_3\) on \(A_\varepsilon\), yields
	\begin{equation}\label{21}
	 \mathbb E[\exp\{cB_3\}\mid A_\varepsilon]\longrightarrow1,
	 \qquad c\in\mathbb R.
	\end{equation}

	4) Estimate of \(B_4\).
	The boundedness of \(DU\) on the localization set gives \(\|Y-\varphi\|_{\infty,*}\le C\varepsilon\) on \(A_\varepsilon\). Since \(x\mapsto\sigma(x)^{-1}J\nabla H(x)\) is Lipschitz there,
	\[
	 \begin{aligned}
	 &\quad\left|\left|\sigma(Y_s)^{-1}J\nabla H(Y_s)\right|^2
	 -\left|\sigma(\varphi_s)^{-1}J\nabla H(\varphi_s)\right|^2\right|\\
	 &\le C\left|
	 \sigma(Y_s)^{-1}J\nabla H(Y_s)
	 -\sigma(\varphi_s)^{-1}J\nabla H(\varphi_s)\right|\\
	 &\le C\varepsilon.
	 \end{aligned}
	\]
	After integration over \([0,T]\),
	\begin{equation*}
	 B_4=-\frac12\int_0^T
	 \left|\sigma(\varphi_s)^{-1}J\nabla H(\varphi_s)\right|^2
	 \,\mathrm ds+O(\varepsilon)
	\end{equation*}
	uniformly on \(A_\varepsilon\).
	Therefore, for every \(c\in\mathbb R\),
	\begin{equation}\label{conditional B4 limit}
	 \begin{aligned}
	 &\lim_{\varepsilon\to0}\mathbb E\left[
	 \exp\left\{c\left(B_4+\frac12\int_0^T
	 \left|\sigma(\varphi_s)^{-1}J\nabla H(\varphi_s)\right|^2
	 \,\mathrm ds\right)\right\}\middle|A_\varepsilon\right]=1.
	 \end{aligned}
	\end{equation}

	Equations \eqref{conditional B1 limit},
	\eqref{conditional B2 limit}, \eqref{21}, and
	\eqref{conditional B4 limit}, followed by Lemma~\ref{lemma 2.9}
	applied to the four centered terms, give
	\begin{equation}\label{23}
	 \begin{aligned}
	 \lim_{\varepsilon\to 0}\mathbb E[\mathcal R\mid A_\varepsilon]
	 &=\exp\left\{
	 \int_0^T\langle\sigma(\varphi_s)^{-1}J\nabla H(\varphi_s),
	 \sigma(\varphi_s)^{-1}\dot\varphi_s\rangle\,\mathrm ds\right.\\
	 &\qquad \qquad \left.
	 -\frac12\int_0^T|\sigma(\varphi_s)^{-1}J\nabla H(\varphi_s)|^2\,\mathrm ds
	 +\frac12\int_0^T(\operatorname{div}\sigma)(\varphi_s)\cdot
	 (\sigma(\varphi_s)^{-1}J\nabla H(\varphi_s))\,\mathrm ds
	 \right\}.
	 \end{aligned}
	\end{equation}

	\textbf{Second step.}
	The exact representation \eqref{9} gives
	\begin{equation*}
	 \mathbb P\bigl(Y\in\mathcal K_U(\varphi,\varepsilon)\bigr)
	 =\mathbb P(\|W-\eta\|_{\infty,*}\le\varepsilon).
	\end{equation*}
	Under \(\widehat{\mathbb P}\), the process \(\hbar=W-\eta\) is Brownian and
	\(D_\varepsilon=\{\|\hbar\|_{\infty,*}\le\varepsilon\}\). Therefore
	\begin{align*}
	 \mathbb P(\|W-\eta\|_{\infty,*}\le\varepsilon)
	 &=\mathbb E_{\widehat{\mathbb P}}
	 [Z_\eta\mathbf1_{D_\varepsilon}]\\
	 &=\exp\{-\|\eta\|_{\mathbb H_0}^2/2\}
	 \mathbb P(\|W\|_{\infty,*}\le\varepsilon)
	 \mathbb E_{\widehat{\mathbb P}}\left[
	 \exp\left\{-\int_0^T
	 \langle\dot\eta_s,\mathrm d\hbar_s\rangle\right\}
	 \middle|D_\varepsilon\right].
	\end{align*}
	Corollary~\ref{corollary vector CM maximum ball} shows that the last
	factor converges to one. Hence
	\begin{equation*}
	 \lim_{\varepsilon\to 0}
	 \frac{\mathbb P(\|W-\eta\|_{\infty,*}\le\varepsilon)}
	 {\mathbb P(\|W\|_{\infty,*}\le\varepsilon)}
	 =\exp\left\{-\frac12\int_0^T|\dot\eta_s|^2\,\mathrm ds\right\}.
	\end{equation*}
	Since \(\dot\eta=\sigma(\varphi)^{-1}\dot\varphi\), this becomes
	\begin{equation}\label{26}
	 \lim_{\varepsilon\to 0}
	 \frac{\mathbb P\bigl(Y\in\mathcal K_U(\varphi,\varepsilon)\bigr)}
	 {\mathbb P(\|W\|_{\infty,*}\le\varepsilon)}
	 =\exp\left\{-\frac12\int_0^T
	 |\sigma(\varphi_s)^{-1}\dot\varphi_s|^2\,\mathrm ds\right\},
	\end{equation}
	where \(\mathbb P(\|W\|_{\infty,*}\le\varepsilon)\) is independent of \(\varphi\).

	\textbf{Last step.}
	By \eqref{tube probability factorization}, the required probability ratio
	is the product
	\[
	 \frac{\mathbb P(X\in\mathcal K_U(\varphi,\varepsilon))}
	 {\mathbb P(\|W\|_{\infty,*}\le\varepsilon)}
	 =\mathbb E[\mathcal R\mid A_\varepsilon]\,
	 \frac{\mathbb P(A_\varepsilon)}{\mathbb P(\|W\|_{\infty,*}\le\varepsilon)}.
	\]
	Both factors have finite limits by \eqref{23} and \eqref{26}. Their
	exponents add, and
	\[
	 \begin{aligned}
	 &\quad \left\langle
	 \sigma(\varphi_s)^{-1}J\nabla H(\varphi_s),
	 \sigma(\varphi_s)^{-1}\dot\varphi_s\right\rangle
	 -\frac12\left|\sigma(\varphi_s)^{-1}J\nabla H(\varphi_s)\right|^2
	 -\frac12\left|\sigma(\varphi_s)^{-1}\dot\varphi_s\right|^2\\
	 &=-\frac12\left|\sigma(\varphi_s)^{-1}
	 \bigl(J\nabla H(\varphi_s)-\dot\varphi_s\bigr)\right|^2.
	 \end{aligned}
	\]
	Consequently,
	\begin{equation*}
	 \lim_{\varepsilon\to 0}
	 \frac{\mathbb P\bigl(X\in\mathcal K_U(\varphi,\varepsilon)\bigr)}
	 {\mathbb P(\|W\|_{\infty,*}\le\varepsilon)}
	 =\exp\left\{-\frac12\mathcal A(\varphi)\right\},
	\end{equation*}
	where \(\mathcal A\) is exactly the functional in \eqref{OM action}.
\end{proof}

\section{Proof of Theorem \ref{T2}}

\begin{proof}[Proof of Theorem \ref{T2}]
	Write \(\sigma=[\sigma_1,\ldots,\sigma_{2n}]\). By \((C4)\),
	\[
	 \sigma_\alpha=J\nabla G_\alpha,\qquad \alpha=1,\ldots,2n.
	\]
	In canonical coordinates \(x=(q_1,\ldots,q_n,p_1,\ldots,p_n)\),
	\[
	 \sigma_\alpha
	 =\left(\partial_{p_1}G_\alpha,\ldots,\partial_{p_n}G_\alpha,
	 -\partial_{q_1}G_\alpha,\ldots,-\partial_{q_n}G_\alpha\right)^\top.
	\]
	Hence equality of the mixed derivatives gives
	\begin{align*}
	 \operatorname{div}\sigma_\alpha
	 =\sum_{j=1}^n
	 \left(\partial_{q_jp_j}^2G_\alpha
	 -\partial_{p_jq_j}^2G_\alpha\right)=0.
	\end{align*}
	Equivalently,
	\(\operatorname{div}\sigma_\alpha=\operatorname{tr}(JD^2G_\alpha)=0\),
	because \(D^2G_\alpha\) is symmetric and \(J\) is skew-symmetric.
	Thus \(\operatorname{div}\sigma=\mathbf 0\), and Theorem~\ref{T1} gives
	\begin{equation*}
	 \mathcal A(\varphi)
	 =\int_0^T\left|\sigma(\varphi_s)^{-1}
	 \bigl(\dot\varphi_s-J\nabla H(\varphi_s)\bigr)\right|^2\,\mathrm ds
	 \ge0.
	\end{equation*}

	Let \(\mathscr A_{x_0}(\mathcal O)\) be the fixed-initial,
	free-terminal admissible class in Definition~\ref{definition OM action}.
	For \(\varphi\in\mathscr A_{x_0}(\mathcal O)\),
	define
	\[
	 R_\varphi(s):=\sigma(\varphi_s)^{-1}
	 \bigl(\dot\varphi_s-J\nabla H(\varphi_s)\bigr).
	\]
	Uniform ellipticity makes \(x\mapsto\sigma(x)^{-1}\) bounded, while
	\(\dot\varphi\in L^2([0,T];\mathbb R^{2n})\), so
	\(R_\varphi\in L^2([0,T];\mathbb R^{2n})\) and
	\[
	 \mathcal A(\varphi)=\|R_\varphi\|_{L^2([0,T];\mathbb R^{2n})}^2.
	\]
	The deterministic orbit \(\varphi^*\) satisfies
	\(\dot\varphi^*=J\nabla H(\varphi^*)\), and therefore
	\(R_{\varphi^*}=0\) and \(\mathcal A(\varphi^*)=0\). Since
	\(\mathcal A\ge0\), this path is a global minimizer on
	{\(\mathscr A_{x_0}(\mathcal O)\)}.

	Conversely, if \({\varphi\in\mathscr A_{x_0}(\mathcal O)}\) is a minimizer, then
	\(\mathcal A(\varphi)=0\). Thus \(R_\varphi=0\) in \(L^2\), and the
	invertibility of \(\sigma(\varphi_s)\) gives
	\[
	 \dot\varphi_s=J\nabla H(\varphi_s)\quad\text{for a.e. }s,
	 \qquad \varphi_0=x_0.
	\]
	Both \(\varphi\) and \(\varphi^*\) satisfy the corresponding integral
	equation. Since \(J\nabla H\) is Lipschitz,
	\[
	 |\varphi_t-\varphi_t^*|
	 \le L\int_0^t|\varphi_s-\varphi_s^*|\,\mathrm ds.
	\]
	Gronwall's inequality yields \(\varphi_t=\varphi_t^*\) for all
	\(t\in[0,T]\). Thus \(\varphi^*\) is the unique action minimizer.

	It remains to translate this variational statement into a path-probability
	statement. For any \({\varphi\in\mathscr A_{x_0}(\mathcal O)}\), Theorem~\ref{T1}
	and the fact that the normalizing small-ball probability is the same for
	all reference paths give
	\[
	 \lim_{\varepsilon\to 0}
	 \frac{\mathbb P(X\in\mathcal K_U(\varphi,\varepsilon))}
	 {\mathbb P(X\in\mathcal K_U(\varphi^*,\varepsilon))}
	 =\exp\left\{-\frac12
	 \bigl(\mathcal A(\varphi)-\mathcal A(\varphi^*)\bigr)\right\}.
	\]
	If \(\varphi\ne\varphi^*\), uniqueness of the zero of the action implies
	\(\mathcal A(\varphi)>0=\mathcal A(\varphi^*)\), so the last limit is
	strictly smaller than one. Therefore \(\varphi^*\) is the unique most
	probable path with respect to the maximum tubes
	\(\mathcal K_U(\cdot,\varepsilon)\)
	in the sense of Definition~\ref{definition MPP}.
\end{proof}

\section{Proof of Theorem \ref{T3}}

\begin{proof}[Proof of Theorem \ref{T3}]
	Set \(b:=J\nabla H\). In It\^o form, equation \eqref{31} is
	\begin{equation}\label{32}
	 \mathrm dX_t^\gamma
	 =\bigl[b(X_t^\gamma)+\gamma^2c(X_t^\gamma)\bigr]\,\mathrm dt
	 +\gamma\sigma(X_t^\gamma)\,\mathrm dW_t,
	\end{equation}
	where
	\begin{equation}\label{33}
	 c(x):=\frac12\sum_{\alpha=1}^{2n}
	 D\sigma_\alpha(x)\sigma_\alpha(x).
	\end{equation}
	Indeed, for the \(i\)-th component,
	\begin{align*}
	 \gamma\sum_{\alpha=1}^{2n}
	 \sigma_{i\alpha}(X_t^\gamma)\circ\mathrm dW_t^\alpha
	 &=
	 \gamma\sum_{\alpha=1}^{2n}
	 \sigma_{i\alpha}(X_t^\gamma)\,\mathrm dW_t^\alpha
	 +\frac{\gamma}{2}\sum_{\alpha=1}^{2n}
	 \mathrm d[\sigma_{i\alpha}(X^\gamma),W^\alpha]_t\\
	 &=
	 \gamma\sum_{\alpha=1}^{2n}
	 \sigma_{i\alpha}(X_t^\gamma)\,\mathrm dW_t^\alpha
	 +\frac{\gamma^2}{2}\sum_{\alpha=1}^{2n}\sum_{j=1}^{2n}
	 \partial_j\sigma_{i\alpha}(X_t^\gamma)
	 \sigma_{j\alpha}(X_t^\gamma)\,\mathrm dt,
	\end{align*}
	which proves \eqref{32}--\eqref{33}. The
	Stratonovich--It\^o correction is of order \(\gamma^2\) and vanishes in
	the limiting controlled equation.

	By hypothesis,
	\[
	 \bigcup_{\varphi\in\overline{\mathbb A}}\varphi([0,T])
	 \subset K\Subset\mathbb D.
	\]
	Choose open sets
	\[
	 K\Subset V_1\Subset V_2\Subset\mathbb D
	\]
	and choose \(\chi\in C_c^\infty(\mathbb R^{2n};[0,1])\) with
	\(\chi=1\) on \(V_1\) and \(\operatorname{supp}\chi\subset V_2\).
	Define on \(\mathbb D\)
	\[
	 \widehat b:=\chi b,\qquad
	 \widehat c:=\chi c,\qquad
	 \widehat\sigma:=\chi\sigma
	\]
	and extend these products by zero to \(\mathbb R^{2n}\). Since
	\(\operatorname{supp}\chi\Subset V_2\Subset\mathbb D\), the zero extensions are globally bounded
	Lipschitz coefficients
	\(\widehat b,\widehat c,\widehat\sigma\) on \(\mathbb R^{2n}\) which agree
	with \(b,c,\sigma\) on \(V_1\). Let \(\widehat X^\gamma\) solve
	\begin{equation}\label{34}
	 \mathrm d\widehat X_t^\gamma
	 =[\widehat b(\widehat X_t^\gamma)
	 +\gamma^2\widehat c(\widehat X_t^\gamma)]\,\mathrm dt
	 +\gamma\widehat\sigma(\widehat X_t^\gamma)\,\mathrm dW_t,
	 \qquad \widehat X_0^\gamma=x_0.
	\end{equation}

	We apply \cite[Theorem~2.1]{62} with its parameter
	\(\varepsilon=\gamma^2\), and verify hypotheses H1--H6 in the notation of
	that theorem. For \(\phi\in C([0,T];\mathbb R^{2n})\), set
	\[
	 b_\varepsilon(t,\phi)
	 :=\widehat b(\phi_t)+\varepsilon\widehat c(\phi_t),
	 \qquad
	 \sigma_\varepsilon(t,\phi):=\widehat\sigma(\phi_t).
	\]
	These maps are predictable and continuous in the stopped path, uniformly
	in \(t\), and \(t\mapsto\widehat\sigma(\phi_t)\) is square integrable.
	Thus H1 holds. Moreover,
	\[
	 \sup_{(t,\phi)}
	 |b_\varepsilon(t,\phi)-\widehat b(\phi_t)|
	 \le\varepsilon\|\widehat c\|_{0;\mathbb R^{2n}},
	 \qquad
	 \sigma_\varepsilon(t,\phi)=\widehat\sigma(\phi_t),
	\]
	which proves H2. The global Lipschitz property of the cutoff coefficients
	gives pathwise uniqueness and strong existence for \eqref{34}; hence H3 is
	satisfied.

	For a deterministic control
	\(u\in L^2([0,T];\mathbb R^{2n})\), the limiting equation is
	\begin{equation}\label{controlled skeleton}
	 \dot\varphi_t=\widehat b(\varphi_t)
	 +\widehat\sigma(\varphi_t)u_t,
	 \qquad \varphi_0=x_0.
	\end{equation}
	The right-hand side is measurable in time and globally Lipschitz in the
	state with integrable Lipschitz coefficient \(C(1+|u_t|)\). Therefore
	\eqref{controlled skeleton} has a unique absolutely continuous solution,
	which defines the skeleton map \(\Gamma_{x_0}(u)=\varphi\) and proves H4.
	If \(\int_0^T|u_s|^2\,\mathrm ds\le R\), boundedness of the coefficients
	gives
	\[
	 \sup_{0\le t\le T}|\varphi_t|
	 \le |x_0|+C(T+\sqrt{TR})
	\]
	and, for \(0\le s\le t\le T\),
	\[
	 |\varphi_t-\varphi_s|
	 \le C\bigl(|t-s|+\sqrt R\,|t-s|^{1/2}\bigr).
	\]
	Consequently, the skeletons generated by the weak \(L^2\)-ball of radius
	\(\sqrt R\) form a relatively compact subset of the path space.

	To verify H5, let \(u_k\rightharpoonup u\) in \(L^2\), with
	\(\int_0^T|u_k(s)|^2\,\mathrm ds\le R\), and write
	\(\varphi^k=\Gamma_{x_0}(u_k)\). The preceding bounds yield a uniformly
	convergent subsequence, say \(\varphi^k\to\varphi\). The control term splits
	as
	\[
	 \begin{aligned}
	 &\int_0^t\widehat\sigma(\varphi_s^k)u_k(s)\,\mathrm ds
	 -\int_0^t\widehat\sigma(\varphi_s)u(s)\,\mathrm ds\\
	 &\quad=
	 \int_0^t\bigl[\widehat\sigma(\varphi_s^k)
	 -\widehat\sigma(\varphi_s)\bigr]u_k(s)\,\mathrm ds
	 +\int_0^t\widehat\sigma(\varphi_s)
	 \bigl(u_k(s)-u(s)\bigr)\,\mathrm ds.
	 \end{aligned}
	\]
	The first integral converges to zero uniformly in \(t\) by the Lipschitz
	bound and Cauchy--Schwarz. For the second one, define
	\[
	 g_{t,j}(s):=\mathbf1_{[0,t]}(s)
	 \widehat\sigma(\varphi_s)^\top e_j,
	 \qquad 0\le t\le T,\quad 1\le j\le2n.
	\]
	For each \(j\), the set \(\{g_{t,j}:0\le t\le T\}\) is compact in
	\(L^2\). Weak convergence is uniform on compact subsets of the dual, so
	\[
	 \sup_{0\le t\le T}\left|
	 \int_0^t\widehat\sigma(\varphi_s)
	 \bigl(u_k(s)-u(s)\bigr)\,\mathrm ds
	 \right|\longrightarrow0.
	\]
	Passing to the limit in the integral equation shows that
	\(\varphi=\Gamma_{x_0}(u)\). Uniqueness identifies every subsequential
	limit; hence \(\Gamma_{x_0}\) is continuous on each weak \(L^2\)-ball,
	as required by H5.

	It remains to verify H6 for varying predictable controls. Let
	\(\varepsilon_k=\gamma_k^2\downarrow0\), and let \(u^k\) be predictable with
	\[
	 \int_0^T|u_s^k|^2\,\mathrm ds\le R
	 \quad\text{almost surely, uniformly in }k.
	\]
	The controlled equation is
	\[
	 \mathrm d\widehat X_t^{\gamma_k,u^k}
	 =\left[\widehat b(\widehat X_t^{\gamma_k,u^k})
	 +\gamma_k^2\widehat c(\widehat X_t^{\gamma_k,u^k})
	 +\widehat\sigma(\widehat X_t^{\gamma_k,u^k})u_t^k\right]\mathrm dt
	 +\gamma_k\widehat\sigma(\widehat X_t^{\gamma_k,u^k})\,\mathrm dW_t.
	\]
	Boundedness of the coefficients, Cauchy--Schwarz, and the
	Burkholder--Davis--Gundy inequality give
	\[
	 \sup_k\mathbb E\left[
	 \sup_{0\le t\le T}|\widehat X_t^{\gamma_k,u^k}|^4\right]
	 \le C_R.
	\]
	The same estimates applied on \([s,t]\) yield
	\[
	 \mathbb E\left[
	 |\widehat X_t^{\gamma_k,u^k}
	 -\widehat X_s^{\gamma_k,u^k}|^4\right]
	 \le C_R|t-s|^2,
	 \qquad 0\le s\le t\le T,
	\]
	uniformly in \(k\). Indeed, the controlled drift contributes at most
	\(C_R|t-s|^2\), while the stochastic integral has fourth moment bounded
	by \(C\gamma_k^4|t-s|^2\). Kolmogorov's tightness criterion therefore
	shows that \(\{\widehat X^{\gamma_k,u^k}\}_k\) is tight in
	\(C([0,T];\mathbb R^{2n})\). Finally,
	\[
	 \sup_k\int_0^T\mathbb E\left[
	 \|\widehat\sigma(\widehat X_s^{\gamma_k,u^k})\|_{\mathrm{HS}}^2
	 \right]\,\mathrm ds
	 \le T\sup_{x\in\mathbb R^{2n}}
	 \|\widehat\sigma(x)\|_{\mathrm{HS}}^2<\infty.
	\]
	This verifies H6. Thus all hypotheses of
	\cite[Theorem~2.1]{62} are satisfied.
	The laws of \(\widehat X^\gamma\) on \(\mathscr C_{x_0}\) consequently
	satisfy an LDP with speed \(\gamma^{-2}\) and good rate function
	\begin{equation}\label{35}
	 \widehat{\mathbb I}(\varphi)
	 =\inf\left\{\frac12\int_0^T|u_s|^2\,\mathrm ds:
	 \begin{array}{l}
	 \varphi\in H^1([0,T];\mathbb R^{2n}),\ \varphi_0=x_0,\\
	 \dot\varphi_s=\widehat b(\varphi_s)
	 +\widehat\sigma(\varphi_s)u_s\ \text{a.e.}
	 \end{array}\right\}.
	\end{equation}
	The limiting skeleton is therefore
	\eqref{controlled skeleton}, with quadratic control cost
	\(\frac12\int_0^T|u_s|^2\,\mathrm ds\).

	For every \(\varphi\in\overline{\mathbb A}\), the image of \(\varphi\) is
	contained in \(K\), where the extended coefficients equal the original
	ones. On this set, \eqref{controlled skeleton} is equivalent to
	\[
	 \sigma(\varphi_s)u_s
	 =\dot\varphi_s-J\nabla H(\varphi_s)
	 \quad\text{for a.e. }s.
	\]
	Condition \((C2)\) makes \(\sigma(\varphi_s)\) invertible, and hence the
	control in \eqref{35} is unique:
	\[
	 u_s=\sigma(\varphi_s)^{-1}
	 \bigl(\dot\varphi_s-J\nabla H(\varphi_s)\bigr).
	\]
	Substitution into the control energy gives
	\[
	 \frac12\int_0^T|u_s|^2\,\mathrm ds
	 =\frac12\int_0^T
	 \left|\sigma(\varphi_s)^{-1}
	 \bigl(\dot\varphi_s-J\nabla H(\varphi_s)\bigr)\right|^2\,\mathrm ds.
	\]
	If \(\varphi\notin H^1([0,T];\mathbb R^{2n})\), no \(L^2([0,T];\mathbb R^{2n})\)-control can satisfy
	\eqref{controlled skeleton}, so the infimum is \(+\infty\). Hence
	\begin{equation}\label{LDP rate identity}
	 \widehat{\mathbb I}(\varphi)=\mathbb I_{\mathbb D}(\varphi),
	 \qquad \varphi\in\overline{\mathbb A}.
	\end{equation}
	Thus \(\widehat{\mathbb I}=\mathbb I_{\mathbb D}\) on the uniformly
	localized set \(\overline{\mathbb A}\); the global LDP may therefore be
	restricted to this set using the local action \(\mathbb I_{\mathbb D}\).

	It remains to identify the relevant events. Couple \(X^\gamma\) and
	\(\widehat X^\gamma\) with the same Brownian motion. The coefficients
	coincide on \(V_1\). Set
	\[
	 \rho^\gamma:=\inf\{t:X_t^\gamma\notin V_1\}.
	\]
	Set likewise
	\[
	 \widehat\rho^\gamma:=\inf\{t:\widehat X_t^\gamma\notin V_1\}.
	\]
	Pathwise uniqueness implies
	\[
	 X_{t\wedge\rho^\gamma\wedge\widehat\rho^\gamma}^\gamma
	 =\widehat X_{t\wedge\rho^\gamma\wedge\widehat\rho^\gamma}^\gamma,
	 \qquad 0\le t\le T,
	\]
	and consequently \(\rho^\gamma=\widehat\rho^\gamma\) on the coupled
	probability space. Since every path in \(\mathbb A\) has its entire range
	in \(K\Subset V_1\), the event \(\{\widehat X^\gamma\in\mathbb A\}\)
	implies \(\widehat\rho^\gamma>T\) and therefore
	\(X^\gamma=\widehat X^\gamma\) on \([0,T]\). Conversely,
	\(\{X^\gamma\in\mathbb A,\tau_{\mathbb D}^\gamma>T\}\) implies
	\(\rho^\gamma>T\) and the same equality of paths. Thus
	\begin{equation}\label{localization equality}
	 \mathbb P\bigl(X^\gamma\in\mathbb A,
	 \tau_{\mathbb D}^\gamma>T\bigr)
	 =\mathbb P(\widehat X^\gamma\in\mathbb A).
	\end{equation}
	Applying the global LDP lower bound to \(\mathbb A^\circ\) and the upper
	bound to \(\overline{\mathbb A}\), then using
	\eqref{LDP rate identity}--\eqref{localization equality}, give
	\[
	-\inf_{\varphi\in \mathbb A^\circ}\mathbb I_{\mathbb D}(\varphi)
	\le
	\liminf_{\gamma\to 0}\gamma^2
	\log\mathbb P\bigl(X^\gamma\in\mathbb A,
	\tau_{\mathbb D}^\gamma>T\bigr)
	\le
	\limsup_{\gamma\to 0}\gamma^2
	\log\mathbb P\bigl(X^\gamma\in\mathbb A,
	\tau_{\mathbb D}^\gamma>T\bigr)
	\le
	-\inf_{\varphi\in\overline{\mathbb A}}\mathbb I_{\mathbb D}(\varphi).
	\]
	This proves the theorem.
\end{proof}

Under \((C4)\), on every uniformly compactly localized admissible class,
\[
\mathcal A(\varphi)
=
\int_0^T
\left|
\sigma(\varphi_s)^{-1}
\bigl(\dot\varphi_s-J\nabla H(\varphi_s)\bigr)
\right|^2\,\mathrm ds
=
2\mathbb I_{\mathbb D}(\varphi).
\]
Hence the fixed-noise small-tube principle and the small-noise deviation
principle are governed by the same quadratic action, up to the normalization
factor \(2\).

\section{Proof of Theorem \ref{T4}}

\begin{proof}[Proof of Theorem \ref{T4}]
	A finite-time deterministic orbit contained in
	\(\mathbb D_{\mathrm{aa}}\) has compact range and therefore has range compactly
	contained in \(\mathbb D_{\mathrm{aa}}\Subset\mathcal O\). By
	Theorem~\ref{T2}, for every \(x_0\in\mathcal O\), the deterministic
	orbit segment
	\[
	 \dot X_t=J\nabla H(X_t),\qquad X_0=x_0,
	\]
	which remains compactly contained in \(\mathcal O\) is the unique most
	probable path with that initial point and a free terminal point.
	We now express this deterministic equation in the action--angle chart.
	Set
	\[
	 \widetilde H(\theta,I):=(H\circ\Psi)(\theta,I)
	 =H_0(I)+P(\theta,I)
	\]
	and let \(J_0=\left(\begin{smallmatrix}0&I_n\\-I_n&0\end{smallmatrix}\right)\)
	denote the standard symplectic matrix in the \((\theta,I)\)-coordinates.
	The symplecticity of \(\Psi\) is equivalent to
	\[
	 D\Psi(\theta,I)J_0D\Psi(\theta,I)^\top=J.
	\]
	Since
	\[
	 \nabla\widetilde H(\theta,I)
	 =D\Psi(\theta,I)^\top\nabla H(\Psi(\theta,I)),
	\]
	we have
	\begin{align*}
	 D\Psi(\theta,I)J_0\nabla\widetilde H(\theta,I)
	 &=D\Psi(\theta,I)J_0D\Psi(\theta,I)^\top
	 \nabla H(\Psi(\theta,I))\\
	 &=J\nabla H(\Psi(\theta,I)).
	\end{align*}
	Consequently, \(X_t=\Psi(\theta_t,I_t)\) solves the deterministic
	Hamiltonian equation in the original phase space if and only if
	\((\theta_t,I_t)\) solves
	\begin{equation}\label{36}
	 \dot\theta=\partial_I(H_0+P)(\theta,I),
	 \qquad
	 \dot I=-\partial_\theta P(\theta,I).
	\end{equation}

	For \(P=0\), equation \eqref{36} reduces to
	\[
	 \dot I=0,\qquad
	 \dot\theta=\omega(I),\qquad
	 \omega(I):=\partial_IH_0(I).
	\]
	Thus every set
	\(\mathcal T_I^0:=\mathbb T^n\times\{I\}\) is invariant and carries the
	rigid rotation
	\[
	 \theta_t=\theta_0+t\omega(I),\qquad I_t=I.
	\]
	The Kolmogorov nondegeneracy means
	\[
	 \det D_I\omega(I)=\det D_I^2H_0(I)\ne0,
	\]
	so the frequency map is locally nondegenerate.
	We now invoke the quantitative result of
	\cite[Theorem~4, Part~I]{6}. Its hypotheses and notation are matched
	explicitly with the present Hamiltonian. Set
	\[
	 \tau:=\nu-1,\qquad
	 \eta:=\|P\|_{C^l(\mathbb T^n\times\mathbb U)},
	 \qquad
	 \mathcal K_0:=\max\{1,\|H_0\|_{C^l(\mathbb U)}\}.
	\]
	The assumptions \(l>2\nu\) and \(\nu>n\) give
	\[
	 l>2\nu=2(\tau+1)>2n\ge4,
	\]
	so the differentiability requirement of the cited
	theorem holds. Define
	\[
	 \mathcal T_0
	 :=\sup_{I\in\mathbb U}
	 \bigl\|(D_I^2H_0(I))^{-1}\bigr\|_2,
	 \qquad
	 \Theta:=\mathcal T_0\mathcal K_0.
	\]
	Both quantities are finite by the hypotheses of Theorem~\ref{T4}. The
	exponent in the smallness condition of
	\cite[Theorem~4, Part~I]{6} is
	\[
	 a:=\frac{1}{l-2\nu}
	 \max\left\{
	 \left(6+\frac{2l}{\nu}\right)(l+\nu)-2l(l-\nu),
	 \frac{2l(l+3\nu)}{\nu}
	 \right\}.
	\]
	Consequently, there exists
	\(c_{\mathrm{KAM}}=c_{\mathrm{KAM}}(n,\tau,l)\in(0,1)\) such that the KAM
	conclusion holds whenever
	\[
	 0<\alpha\le c_{\mathrm{KAM}}\mathcal K_0
	\]
	and
	\[
	 \eta\le
	 c_{\mathrm{KAM}}
	 \mathcal K_0^{-(l+2\nu)/(l-2\nu)}
	 \Theta^{-a}\alpha^{2l/(l-2\nu)}.
	\]
	These inequalities are compatible with the asymptotic choice stated in
	Theorem~\ref{T4}. The case \(\eta=0\) is immediate because the
	Hamiltonian is then integrable, so assume \(\eta>0\). Choose
	\(C_*>0\) so that
	\[
	 c_{\mathrm{KAM}}
	 \mathcal K_0^{-(l+2\nu)/(l-2\nu)}
	 \Theta^{-a}
	 C_*^{2l/(l-2\nu)}\ge1
	\]
	and set
	\[
	 \alpha:=C_*
	 \eta^{(l-2\nu)/(2l)}
	 =C_*\eta^{1/2-\nu/l}.
	\]
	The second KAM inequality then follows identically, whereas the first
	holds for all sufficiently small \(\eta\). The comparison constant \(C_*\)
	depends only on \(n,\nu,l,\mathcal K_0\), and \(\mathcal T_0\), and hence
	on the fixed integrable Hamiltonian and action domain, not on \(\eta\).

	For \(r>0\), let \(B_r(I)\) be the open ball in the maximum norm of
	\(\mathbb R^n\), and put
	\[
	 B_r(A):=\bigcup_{I\in A}B_r(I),
	 \qquad A\subset\mathbb R^n.
	\]
	Define
	\[
	 \Delta_\alpha^\tau
	 :=\left\{\omega\in\mathbb R^n:
	 |\omega\cdot k|\ge\frac{\alpha}{|k|_1^\tau}
	 \text{ for every }k\in\mathbb Z^n\setminus\{0\}\right\}.
	\]
	With
	\[
	 \alpha_*:=\alpha^{1/(l-2\nu)},
	\]
	the interior action domain and its Diophantine subset are
	\[
	 \mathbb U'
	 :=\{I\in\mathbb U:B_{\alpha_*}(I)\subset\mathbb U\}
	\]
	and
	\[
	 \mathbb U_\alpha
	 :=\{I\in\mathbb U':
	 \partial_IH_0(I)\in\Delta_\alpha^\tau\}.
	\]
	The interior restriction prevents the KAM deformation from reaching
	the boundary of the action domain.
	We verify that the parameter set is nonempty for all sufficiently small
	positive \(\eta\). Choose \(\overline I\in\mathbb U\) and \(r>0\) such
	that \(B_{2r}(\overline I)\Subset\mathbb U\). Since
	\(\alpha_*=\alpha^{1/(l-2\nu)}\to0\), one has
	\(B_r(\overline I)\subset\mathbb U'\) for sufficiently small \(\eta\).
	The inverse function theorem and
	\(\det D_I^2H_0(\overline I)\ne0\) provide an open ball
	\(B_\delta(\overline I)\Subset B_r(\overline I)\) on which the frequency
	map \(I\mapsto\partial_IH_0(I)\) is a diffeomorphism onto an open set.
	Because \(\tau=\nu-1>n-1\), the resonant-slab estimate on every bounded
	frequency ball \(V\) gives
	\[
	 \operatorname{Leb}\bigl(V\setminus\Delta_\alpha^\tau\bigr)
	 \le C_V\alpha
	 \sum_{k\in\mathbb Z^n\setminus\{0\}}
	 |k|_1^{-(\tau+1)}.
	\]
	The series converges. Taking a ball
	\(V\Subset\partial_IH_0(B_\delta(\overline I))\) and then decreasing
	\(\eta\), hence \(\alpha\), shows that
	\(V\cap\Delta_\alpha^\tau\ne\varnothing\). Its inverse image belongs to
	\(\mathbb U_\alpha\). Thus the Cantor family constructed below is not
	empty. When \(\eta=0\), the unperturbed invariant tori already give the
	asserted family, independently of this argument.

	By \cite[Theorem~4, Part~I]{6}, there exist a Cantor-like set
	\(\mathbb U_*\subset\mathbb U\), a Whitney \(C^2\) function
	\(H_*:\mathbb U_*\to\mathbb R\), and an embedding
	\[
	 \Phi_*:\mathbb U_*\times\mathbb T^n
	 \longrightarrow\mathbb T^n\times\mathbb U
	\]
	whose restriction to each torus is of class \(C^\nu\), such that
	\[
	 \widetilde H\bigl(\Phi_*(I_*,\vartheta)\bigr)=H_*(I_*),
	 \qquad
	 (I_*,\vartheta)\in\mathbb U_*\times\mathbb T^n.
	\]
	The frequency-preserving parameter map
	\[
	 G^*:=(\partial_{I_*}H_*)^{-1}\circ\partial_IH_0:
	 \mathbb U_\alpha\longrightarrow\mathbb U_*
	\]
	is a bi-Lipschitz homeomorphism onto \(\mathbb U_*\) and satisfies
	\[
	 \partial_{I_*}H_*\bigl(G^*(I)\bigr)
	 =\partial_IH_0(I),
	 \qquad I\in\mathbb U_\alpha.
	\]
	In particular, the label of a torus in the perturbed Cantor set is not
	the unperturbed action itself but its image under \(G^*\). The theorem
	also gives
	\[
	 B_{\alpha_*/2}(\mathbb U_*)\subset\mathbb U
	\]
	and the displacement bound
	\[
	 \begin{aligned}
	 \|G^*-\operatorname{id}\|_{0;\mathbb U_\alpha}
	 &\le
	 \eta^{3\tau/[2(l+\nu)]}
	 \alpha^{(l+1)/(l+\nu)}
	 \mathcal K_0^{-\tau/[2(l+\nu)]}\\
	 &\quad{}\times
	 \Theta^{-1-2l\tau/[\nu(l+\nu)]},
	 \end{aligned}
	\]
	together with
	\[
	 \operatorname{Lip}_{\mathbb U_\alpha}
	 (G^*-\operatorname{id})<\frac12.
	\]
	These estimates, together with the embedding estimate in
	\cite[Theorem~4, estimate~(6)]{6}, quantify the assertion that the
	persistent tori are small deformations of the corresponding
	unperturbed tori.

	Fix \(I_0\in\mathbb U_\alpha\), and set
	\[
	 I_*:=G^*(I_0),
	 \qquad
	 \omega:=\partial_IH_0(I_0).
	\]
	The definition of \(\mathbb U_\alpha\) gives
	\[
	 |k\cdot\omega|
	 \ge\frac{\alpha}{|k|_1^\tau},
	 \qquad
	 k\in\mathbb Z^n\setminus\{0\},
	\]
	while the frequency-preserving identity gives
	\[
	 \partial_{I_*}H_*(I_*)=\partial_IH_0(I_0)=\omega.
	\]
	Define the perturbed torus embedding by
	\[
	 K_{I_0,\eta}(\vartheta)
	 :=\Phi_*(I_*,\vartheta)
	 =\Phi_*\bigl(G^*(I_0),\vartheta\bigr).
	\]
	It follows from \cite[Remark~5(i)]{6} that
	\(K_{I_0,\eta}(\mathbb T^n)\) is an invariant Lagrangian Kronecker
	torus of \(\widetilde H\), and the Hamiltonian flow satisfies the exact
	conjugacy relation
	\[
	 \Phi_{\widetilde H}^t\circ K_{I_0,\eta}
	 =K_{I_0,\eta}\circ R_\omega^t,
	 \qquad
	 R_\omega^t(\vartheta):=\vartheta+t\omega.
	\]
	The inclusion
	\(B_{\alpha_*/2}(\mathbb U_*)\subset\mathbb U\) and the range of
	\(\Phi_*\) ensure that
	\[
	 K_{I_0,\eta}(\mathbb T^n)
	 \Subset\mathbb T^n\times\mathbb U.
	\]

	Since \(\nu>n\ge2\), the embedding \(K_{I_0,\eta}\) is at least
	\(C^1\). Differentiating the flow conjugacy at \(t=0\) is therefore
	legitimate and yields the KAM invariance equation
	\begin{equation}\label{KAM invariance equation}
	 DK_{I_0,\eta}(\vartheta)\omega
	 =J_0\nabla\widetilde H(K_{I_0,\eta}(\vartheta)).
	\end{equation}
	For every \(\vartheta_0\in\mathbb T^n\), let
	\[
	 y_t:=K_{I_0,\eta}(\vartheta_0+t\omega).
	\]
	By the chain rule and \eqref{KAM invariance equation},
	\[
	 \dot y_t
	 =DK_{I_0,\eta}(\vartheta_0+t\omega)\omega
	 =J_0\nabla\widetilde H(y_t).
	\]
	Thus \(y_t\) is a trajectory of the deterministic Hamiltonian flow in
	action--angle coordinates. Mapping it to the original phase space by
	\[
	 X_t^{I_0,\eta}:=\Psi(y_t),
	\]
	and using the symplectic identity established above, we obtain
	\begin{align*}
	 \dot X_t^{I_0,\eta}
	 &=D\Psi(y_t)\dot y_t\\
	 &=D\Psi(y_t)J_0\nabla\widetilde H(y_t)\\
	 &=J\nabla H(\Psi(y_t))\\
	 &=J\nabla H(X_t^{I_0,\eta}).
	\end{align*}
	Hence
	\[
	 \mathcal T_{I_0,\eta}
	 :=\Psi\bigl(K_{I_0,\eta}(\mathbb T^n)\bigr)
	 \Subset\mathbb D_{\mathrm{aa}}\Subset\mathcal O
	\]
	is invariant under the deterministic Hamiltonian flow. Since both
	\(K_{I_0,\eta}(\mathbb T^n)\) and \(\Psi\) are Lagrangian and
	symplectic, respectively, \(\mathcal T_{I_0,\eta}\) is a Lagrangian
	torus. The conjugacy above shows that its internal motion is the
	quasi-periodic rotation with the unchanged Diophantine frequency
	\(\omega\).

	For every finite \(T>0\), the orbit
	\(X^{I_0,\eta}|_{[0,T]}\) has compact range in \(\mathcal O\) and is
	admissible in Theorem~\ref{T2}. Condition \((C4)\) reduces the
	Onsager--Machlup action to its nonnegative quadratic part. Since
	\(X^{I_0,\eta}\) satisfies the deterministic Hamiltonian equation,
	\[
	 \mathcal A(X^{I_0,\eta})
	 =\int_0^T
	 \left|\sigma(X_s^{I_0,\eta})^{-1}
	 \bigl(\dot X_s^{I_0,\eta}
	 -J\nabla H(X_s^{I_0,\eta})\bigr)\right|^2\,\mathrm ds
	 =0.
	\]
	Conversely, if
	\(\varphi\in\mathscr A_{X_0^{I_0,\eta}}(\mathcal O)\) and
	\(\mathcal A(\varphi)=0\), then
	\[
	 \dot\varphi_s=J\nabla H(\varphi_s)
	 \quad\text{for almost every }s\in[0,T].
	\]
	Because \(\varphi_0=X_0^{I_0,\eta}\), uniqueness for the deterministic
	Hamiltonian equation implies
	\(\varphi=X^{I_0,\eta}\). Therefore every other admissible path with
	the same initial point has strictly positive action. Applying
	Theorem~\ref{T1} to both tube probabilities gives
	\[
	 \lim_{\varepsilon\to0}
	 \frac{\mathbb P(X\in\mathcal K_U(\varphi,\varepsilon))}
	 {\mathbb P(X\in\mathcal K_U(X^{I_0,\eta},\varepsilon))}
	 =
	 \exp\left\{-\frac12
	 \bigl[\mathcal A(\varphi)-\mathcal A(X^{I_0,\eta})\bigr]\right\}.
	\]
	Using \(\mathcal A(X^{I_0,\eta})=0\), the right-hand side becomes
	\[
	 \exp\left\{-\frac12\mathcal A(\varphi)\right\}<1
	\]
	unless \(\varphi=X^{I_0,\eta}\). Thus every
	finite-time deterministic orbit on a surviving KAM torus is the unique
	center of the most-probable \(U\)-coordinate tubes with its initial point
	fixed.

	Finally, after replacing \(\sigma\) by
	\(\gamma\sigma\), Theorem~\ref{T3} applies with speed
	\(\gamma^{-2}\) and localized action
	\[
	 \mathbb I_{\mathbb D}(\varphi)
	 =\frac12\int_0^T
	 \left|\sigma(\varphi_s)^{-1}
	 \bigl(\dot\varphi_s-J\nabla H(\varphi_s)\bigr)\right|^2\,\mathrm ds
	\]
	on the admissible class. Under \((C4)\),
	\(\mathcal A(\varphi)=2\mathbb I_{\mathbb D}(\varphi)\), completing the
	connection between the most probable KAM dynamics and the small-noise
	deviation principle.
\end{proof}

\begin{example}
	We give an explicit state-dependent diffusion satisfying
	\((C1)\)--\((C4)\). Let \(n=2\), write
	\(x=(q_1,q_2,p_1,p_2)\), and fix
	\(0<I_j^-<I_j^+<\infty\). Consider the product annulus
	\[
	\mathbb D
	 =\left\{(q,p)\in\mathbb R^4:
	 I_j^-<\frac{q_j^2+p_j^2}{2}<I_j^+,\ j=1,2\right\}.
	\]
	Choose bounded open rectangles
	\[
	 \mathbb U_{\mathrm{aa}}\Subset\mathbb U_{\mathcal O}
	 \Subset(I_1^-,I_1^+)\times(I_2^-,I_2^+).
	\]
	The action--angle map
	\[
	 \Psi(\theta,I)=
	 \bigl(\sqrt{2I_1}\cos\theta_1,
	 \sqrt{2I_2}\cos\theta_2,
	 -\sqrt{2I_1}\sin\theta_1,
	 -\sqrt{2I_2}\sin\theta_2\bigr)
	\]
	is a diffeomorphism from
	\(\mathbb T^2\times((I_1^-,I_1^+)\times(I_2^-,I_2^+))\) onto
	\(\mathbb D\). Direct calculation gives
	\(\Psi^*(\mathrm dq_1\wedge\mathrm dp_1+
	\mathrm dq_2\wedge\mathrm dp_2)=
	\mathrm d\theta_1\wedge\mathrm dI_1+
	\mathrm d\theta_2\wedge\mathrm dI_2\), so \(\Psi\) is symplectic. Set
	\[
	 \mathcal O:=\Psi(\mathbb T^2\times\mathbb U_{\mathcal O}),
	 \qquad
	 \mathbb D_{\mathrm{aa}}:=\Psi(\mathbb T^2\times\mathbb U_{\mathrm{aa}});
	\]
	then \(\mathbb D_{\mathrm{aa}}\Subset\mathcal O\Subset\mathbb D\).
	For nonzero constants \(\kappa_1,\kappa_2\), define
	\[
	 \sigma(q,p)=
	 \begin{pmatrix}
	 \kappa_1p_1&0&1&0\\
	 0&\kappa_2p_2&0&1\\
	 1&0&0&0\\
	 0&1&0&0
	 \end{pmatrix}.
	\]
	The matrix is symmetric and genuinely state dependent.

	Define
	\[
	 U(y_1,y_2,y_3,y_4)
	 =\left(y_3+\frac{\kappa_1}{2}y_1^2,\
	 y_4+\frac{\kappa_2}{2}y_2^2,\ y_1,\ y_2\right).
	\]
	Its inverse is
	\[
	 U^{-1}(q_1,q_2,p_1,p_2)
	 =\left(p_1,p_2,\
	 q_1-\frac{\kappa_1}{2}p_1^2,\
	 q_2-\frac{\kappa_2}{2}p_2^2\right).
	\]
	Set
	\(\widetilde{\mathcal O}:=U^{-1}(\mathcal O)\). Direct differentiation
	gives
	\[
	 DU(y)=
	 \begin{pmatrix}
	 \kappa_1y_1&0&1&~0~\\
	 0&\kappa_2y_2&0&~1~\\
	 1&0&0&~0~\\
	 0&1&0&~0~
	 \end{pmatrix}
	 =\sigma(U(y)),
	\]
	which verifies \((C3)\).

	With \(J=\left(\begin{smallmatrix}0&I_2\\-I_2&0\end{smallmatrix}\right)\),
	let
	\[
	 G_1=\frac{\kappa_1}{2}p_1^2-q_1,\qquad
	 G_2=\frac{\kappa_2}{2}p_2^2-q_2,\qquad
	 G_3=p_1,\qquad G_4=p_2.
	\]
	For example,
	\[
	 J\nabla G_1=(\kappa_1p_1,0,1,0)^\top=\sigma_1,
	 \qquad
	 J\nabla G_3=(1,0,0,0)^\top=\sigma_3,
	\]
	and the remaining two columns are checked in the same way. Hence
	\(\sigma_\alpha=J\nabla G_\alpha\) for
	\(\alpha=1,\ldots,4\), so \((C4)\) holds and
	\(\operatorname{div}\sigma=\mathbf 0\).

	The inverse matrix is
	\[
	 \sigma(q,p)^{-1}=
	 \begin{pmatrix}
	 ~0~&0&1&0\\
	 ~0~&0&0&1\\
	 ~1~&0&-\kappa_1p_1&0\\
	 ~0~&1&0&-\kappa_2p_2
	 \end{pmatrix}.
	\]
	Since \(\mathbb D\) is bounded,
	\[
	 \sup_{x\in\overline{\mathbb D}}\|\sigma(x)\|_2<\infty,
	 \qquad
	 \sup_{x\in\overline{\mathbb D}}\|\sigma(x)^{-1}\|_2<\infty.
	\]
	Thus the singular values of \(\sigma\) have
	uniform positive lower and finite upper bounds, and there are
	\(0<\lambda\le\Lambda\) such that
	\[
	 \lambda|v|^2\le v^\top\sigma(x)\sigma(x)^\top v
	 \le\Lambda|v|^2,\qquad x\in\overline{\mathbb D}.
	\]
	This verifies \((C2)\).

	In the symplectic action--angle chart \(\Psi\), take
	\[
	 H_0(I)=\omega_1I_1+\omega_2I_2+\frac12(I_1^2+I_2^2),
	 \qquad
	 P(\theta,I)=\delta\cos(\theta_1-\theta_2),
	 \qquad H\circ\Psi=H_0+P.
	\]
	Because \(I_j^->0\),
	\[
	 \cos(\theta_1-\theta_2)
	 =\frac{q_1q_2+p_1p_2}
	 {\sqrt{q_1^2+p_1^2}\sqrt{q_2^2+p_2^2}}
	\]
	is smooth on a neighborhood of \(\overline{\mathbb D}\). Hence \((C1)\) follows from smoothness and compactness. Moreover, \(D_I^2H_0=I_2\), so \(H_0\) is Kolmogorov nondegenerate. For sufficiently small \(|\delta|\), Theorem~\ref{T4} applies on \(\mathbb D_{\mathrm{aa}}\Subset\mathcal O\): the Diophantine invariant tori persist, up to a slight deformation, in the most-probable-path sense. Replacing \(\sigma\) by \(\gamma\sigma\) gives the local large deviation bounds of Theorem~\ref{T3}.
\end{example}

\bibliographystyle{plain}
\bibliography{Ref}
\end{document}